\documentclass[11pt,reqno,twoside]{amsart}
\usepackage{graphicx}
\usepackage{amssymb}
\usepackage{epstopdf}
\usepackage[asymmetric,top=3.5cm,bottom=4.3cm,left=3.1cm,right=3.1cm]{geometry}
\usepackage{graphicx}
\usepackage{bm}
\usepackage{booktabs} 
\usepackage{array} 
\usepackage{paralist} 
\usepackage{verbatim} 
\usepackage{subfig} 
\usepackage{tabularx}
\usepackage{amsmath,amsfonts,amsthm,mathrsfs,amssymb,cite}
\usepackage[usenames]{color}
\usepackage{ulem}

\usepackage{graphics}
\usepackage{framed}

\usepackage{enumerate}
\usepackage{hyperref}
\usepackage{xcolor}
\hypersetup{
  colorlinks,
  linkcolor={blue!80!black},
  urlcolor={blue!80!black},
  citecolor={blue!80!black}
}
\usepackage[capitalize,noabbrev]{cleveref}

\newcommand\px{\mathrel{/\mkern-5mu/}}
\newtheorem{thm}{Theorem}[section]

\newtheorem{lem}{Lemma}[section]
\newtheorem{prop}{Proposition}[section]
\theoremstyle{definition}
\newtheorem{defn}{Definition}[section]
\theoremstyle{remark}

\numberwithin{equation}{section}

\newcommand{\Kcal}{\mathcal{K}}
\newcommand{\Scal}{\mathcal{S}}
\newcommand{\Dcal}{\mathcal{D}}
\newcommand{\Acal}{\mathcal{A}}
\newcommand{\Rcal}{\mathcal{R}}
\newcommand{\Mcal}{\mathcal{M}}
\newcommand{\Ncal}{\mathcal{N}}
\newcommand{\Hcal}{\mathcal{H}}
\newcommand{\Tcal}{\mathcal{T}}
\newcommand{\Pcal}{\mathcal{P}}
\newcommand{\Bcal}{\mathcal{B}}

\newcommand{\Ccal}{\mathcal{C}}
\newcommand{\Fcal}{\mathcal{F}}

\newcommand{\bu}{\mathbf{u}}
\newcommand{\bC}{\mathbf{C}}
\newcommand{\bv}{\mathbf{v}}
\newcommand{\bw}{\mathbf{w}}
\newcommand{\bff}{\mathbf{f}}

\newcommand{\bU}{\mathbf{U}}

\newcommand{\bd}{\mathbf{d}}
\newcommand{\bx}{\mathbf{x}}
\newcommand{\by}{\mathbf{y}}

\newcommand{\bg}{\mathbf{g}}

\newcommand{\bp}{\mathbf{p}}

\newcommand{\Bx}{\mathbf{x}}

\newcommand{\beq}{\begin{equation}}
\newcommand{\eeq}{\end{equation}}

\title[Stable determination in inverse elastic scattering and beyond]
{Stable determination of an elastic medium scatterer by a single far-field measurement and beyond}
\author{Zhengjian Bai}
\address{School of Mathematical Sciences, Xiamen University, Xiamen 361005, China}
\email{zjbai@xmu.edu.cn}

\author{Huaian Diao}
\address{School of Mathematics, Jilin University,  Changchun 130012, China}
\email{diao@jlu.edu.cn}

\author{Hongyu Liu}
\address{Department of Mathematics, City University of Hong Kong, Kowloon, Hong Kong, China}
\email{hongyu.liuip@gmail.com; hongyuliu@hkbu.edu.hk}

\author{Qingle Meng}
\address{School of Mathematical Sciences, Xiamen University, Xiamen 361005, China}
\email{mengql2021@foxmail.com}

\date{} 
\begin{document}
\maketitle

\begin{abstract}

We are concerned with the time-harmonic elastic scattering due to an inhomogeneous elastic material inclusion located inside a uniformly homogeneous isotropic medium. We establish a sharp stability estimate of logarithmic type in determining the support of the elastic scatterer, independent of its material content, by a single far-field measurement when the support is a convex polyhedral domain in $\mathbb{R}^n$, $n=2,3$. Our argument in establishing the stability result is localized around a corner of the medium scatterer. This enables us to further establish a byproduct result by proving that if a generic medium scatterer, not necessary to be a polyhedral shape, possesses a corner, then there exists a positive lower bound of the scattered far-field patterns. The latter result indicates that if an elastic material object possesses a corner on its support, then it scatters every incident wave stably and invisibility phenomenon does not occur.

\medskip



\end{abstract}

\section{Introduction}\label{sec1}

\subsection{Mathematical setup}

We are mainly concerned with the time-harmonic elastic wave scattering due to the impingement of an incident field on an inhomogeneous isotropic medium scatterer as well as the associated inverse problem of determining the scatterer from the corresponding far-field measurement. We first introduce the mathematical formulation of our study.

Let $\Omega$ be a bounded Lipschitz domain in $\mathbb{R}^n$ with a connected complement $\mathbb{R}^n\backslash\overline{\Omega}$, $n=2,3$. In the physical setup, $\Omega$ is the support of an inhomogeneous elastic scatterer embedded in a uniformly homogeneous background space. The elastic medium parameters are characterised by $\rho$ and $\lambda, \mu$, which are respectively referred to as the density and the bulk moduli. It is assumed that $\lambda$ and $\mu$ are real constants satisfying the strong convexity conditions:
$$
\mu>0  \quad\mbox{ and }\quad n\,\lambda+2\mu>0.
$$
It is also assumed that $\rho\in L^\infty(\mathbb{R}^n)$ and ${\Omega=\mathrm{supp}(1-\rho)}$.  That is, by normalisation, we assume that the density of the homogeneous background space is $1$, whereas the scatterer is characterised by the inhomogeneous density $\rho$ in $\Omega$.

To introduce the Lam\'e system that describes the elastic scattering, we let $\omega\in\mathbb{R}_+$ signify the angular frequency of the time-harmonic elastic wave. Denote
$$
\kappa_{\mathrm s}=\omega\sqrt{1/\mu}  \mbox{ and } \kappa_{\mathrm p}=\omega\sqrt{1/(\lambda+2\mu)}
$$
by the shear and compressional wave numbers, respectively.  Let $\bu^i$ denote the time-harmonic plane incident wave of the following form:
\begin{align}\label{eq:incident wave}
\bu^i=\alpha_1\, \bd \,e^{\imath\kappa_{\mathrm p} \bx\cdot \bd} +\alpha_2\,\bd^\perp e^{\imath\kappa_{\mathrm s} \bx\cdot \bd}, \quad  \alpha_1, \alpha_2\in \mathbb{C},\quad |\alpha_1|+|\alpha_2|\neq 0,
\end{align}
where $\bd\in \mathbb{S}^{n-1}$ is called the incident direction, and $\bd^\perp$ is a unit vector orthogonal to $\bd$. Due to the interaction between the incident elastic wave $\bu^i$ and the elastic scatter $\Omega$, the scattered elastic wave $\bu^s \in \mathbb C^n$ is generated. This induces the total elastic wave $\bu \in \mathbb C^n$, which is the superposition of $\bu^i$ and $\bu^s$, namely $\bu:=\bu^i+\bu^s$, and satisfies the Navier equation
\begin{equation}\label{eq:navier equation}
\Delta^* \bu+\rho \,\omega^2 \bu={\bf 0},\quad \Delta^*\bu:=\mu\Delta\bu+(\lambda+\mu)\nabla(\nabla\cdot\bu),\quad\mbox{in}\ \mathbb{R}^n\backslash\overline{\Omega}.
\end{equation}
Clearly, $\bu^i$ is an entire solution to \eqref{eq:navier equation} with $\rho=1$. By the Helmholtz decomposition, any solution $\bu \in H_{loc}^2(\mathbb{R}^n)^n$ to \eqref{eq:navier equation} can be decomposed as follows:
\begin{equation}\label{eq:helmoltz decomposition1-1}
\bu=\bu_\mathrm p+\bu_\mathrm s,
\end{equation}
where $\bu_{\mathrm p}$ and $\bu_{\mathrm s}$ satisfy the equations:
 \begin{equation}\label{eq:helmoltz decomposition1-2}
 \begin{cases}
(\Delta+\kappa^2_\mathrm p)\bu_\mathrm p={\bf 0}, &\quad\nabla \times \bu_\mathrm p=\bf 0,\medskip\\
(\Delta+\kappa^2_\mathrm s)\bu_\mathrm s={\bf 0}, &\quad \nabla \,\cdot\, \bu_\mathrm s=0.
\end{cases}
\end{equation}
 The scattered wave $\bu^s$ satisfies the Kurpradze radiation condition
\begin{equation}\label{eq:rd1}
 \begin{cases}
\lim_{|\bx| \to\infty}|\bx|^{(n-1)/2}(  \frac{\partial{\bu_{\mathrm p}^s}}{\partial{|\bx|}}-\imath\kappa_p \bu_{\mathrm p}^s )=\bf 0,\medskip\\
\lim_{|\bx| \to\infty}|\bx|^{(n-1)/2}(  \frac{\partial{\bu_{\mathrm s}^s}}{\partial{|\bx|}}-\imath\kappa_s \bu_{\mathrm s}^s )=\bf 0,
\end{cases}
\end{equation}
 and admits the following asymptotic expansion (cf. \cite{Haher1998}):
\begin{align}\label{ineq:farfiel}
\bu^s(\bx)=\dfrac{\exp(\imath\kappa_{\mathrm p} |\bx|)}{|\bx|^{(n-1)/2}}\bU_{\mathrm p}(\hat {\bx})+\dfrac{\exp(\imath\kappa_{\mathrm s}|\bx|)}{|\bx|^{(n-1)/2}}\bU_{\mathrm s}(\hat {\bx})+\mathcal{O}(|\bx|^{-\frac{n+1}{2}})\quad \mbox{as}\quad |\bx|\to\infty,
\end{align}
which holds uniformly in all directions $\hat{\bx}\in \mathbb{S}^{n-1}$, where $\imath:=\sqrt{-1}$ is the imaginary unit. The vector fields $\bU_\mathrm p$ and $\bU_\mathrm s$ are referred to as the longitudinal and transversal far-field patterns, respectively, which respectively characterize the asymptotic behaviors of the normal part $\bu^s_{\mathrm p}$ and of the tangential part $\bu^s_{\mathrm s}$ of $\bu^s$. Let $\bU=(\bU_\mathrm p,\,\bU_\mathrm s)$ denote the far-field pattern of $\bu^s$, which is an element of the liner space $L^2(\mathbb{S}^{n-1})^n\times L^2(\mathbb{S}^{n-1})^n$ equipped with the following norm
  $$
  \|(\bw_1,\bw_2)\|^2_{L^2(\mathbb{S}^{n-1},\mathbb{C}^n\times \mathbb{C}^n)}=\|\bw_1\|^2_{L^2(\mathbb{S}^{n-1})^n}+\|\bw_2\|^2_{L^2(\mathbb{S}^{n-1})^n},
  $$
where $(\bw_1,\bw_2)\in L^2(\mathbb{S}^{n-1})^n\times L^2(\mathbb{S}^{n-1})^n $.

The inverse problem of our study can be introduced as determining the support of the inhomogeneous scattering object, namely $\Omega$, independent of its material content $(\lambda, \mu, \rho)$, by knowledge of the far-field pattern $\mathbf{U}(\hat{\bx})$. By introducing an operator $\mathscr{F}$ which is defined by the forward scattering system \eqref{eq:incident wave}--\eqref{ineq:farfiel} and sends the scatterer $(\Omega; \lambda,\mu,\rho)$ to the corresponding far-field pattern associated with an incident field $\mathbf{u}^i$, the inverse problem can be abstractly given as:
\begin{equation}\label{eq:ip1}
\mathscr{F}(\Omega)=\mathbf{U}(\hat{\bx}; \mathbf{u}^i),\quad \hat{\bx}\in\mathbb{S}^{n-1}.
\end{equation}
It is emphasized that in our study of \eqref{eq:ip1}, the material content of $\Omega$, namely $(\lambda,\mu, \rho)$, is not required to be known in advance though it belongs to a certain general a-priori class as shall be introduced in what follows. Moreover, we shall consider the case that only a fixed $\mathbf{u}^i$ is used, namely $\mathbf{u}^i$ is given in \eqref{eq:incident wave} with fixed $\eta_1,\eta_2,\mathbf{d}$ and $\mathbf{d}^\perp$. In such a case, $\mathbf{U}(\hat\bx; \mathbf{u}^i)$ is said to be {a} single far-field measurement. The inverse scattering problem with a single far-field measurement constitutes a longstanding problem in the literature. Finally, we remark that by direction verifications, the inverse problem \eqref{eq:ip1} is nonlinear and it is formally-determined with a single far-field measurement.

The second focus of our study is under what conditions, $\mathbf{U}\equiv\mathbf{0}$. This is another aspect of the inverse scattering problem \eqref{eq:ip1} and in such a case, the underlying object $(\Omega;\lambda,\mu,\rho)$ is invisible with respect to the far-field measurement.

\subsection{Statement of the main stability result and its implication to invisibility }\label{sec:main results}

First, we introduce the a-priori class of elastic scatterers for our study.


\begin{defn}\label{defin:admissible class}
Let $\Omega$ be a bounded Lipschitz domain in $\mathbb{R}^n$ with a connected complement $\mathbb{R}^n\backslash\overline{\Omega}$, $n=2,3$. In addition, $\Omega=\mathrm{supp}(1-\rho)\subset \overline{B_R}$, where $B_R$ signifies a central ball of radius $R$.
We say $(\Omega; \rho)\in\Kcal$ if the following conditions are satisfied:
\begin{itemize}
\item[{\rm (a)}] $\Omega\subset \mathbb{R}^{2}$ is a convex polygon and the opening angle at each vertex of $\Omega$ is in $(2\alpha_m,2\alpha_M)$, $\alpha_m>0$ and $\alpha_M<\pi/2$;
\item[{\rm (b)}] $\Omega\subset \mathbb{R}^{3}$ is a convex polyhedron;
\item[{\rm (c)}] The distances of any vertex of $\Omega$ to its non-adjacent edges are at least $l_0$, $0<l_0\leq 1$;
\item[{\rm (d)}] $\rho(\bx)$ is a uniformly $\theta$-H\"{o}lder continuous function in $\overline\Omega$, $0<\theta\leq1$. In addition, $|\rho(\bx_0)-1| \geq\epsilon_0>0$ at any vertex $\bx_0$ of $\Omega$.
\end{itemize}
\end{defn}
\begin{defn}\label{defin:admissible well-possed and non-vainishing} We say $\rho$ is called an admissible density function if for a time-harmonic plane incident wave $\bu^i$ of the form \eqref{eq:incident wave}, the forward scattering system \eqref{eq:incident wave}--\eqref{eq:rd1} admits a unique solution $\bu\in H_{loc}^2(\mathbb{R}^n)^n$ such that the scattered wave $\bu^s=\bu-\bu^i$ is the radiating solution and $\|\bu^s \|_{H^2(B_{2R})^n}\leq \Ncal$, where $\Ncal$ is an a-priori positive constant and $\bu(\bx)\neq 0$ in $S_\Omega\cup B_R\backslash \overline{\Omega}$, where $S_\Omega$ denotes the set of vertices of $\Omega$.

\end{defn}


In what follows, $(\Omega, \rho)$ is said to be an admissible polyhedral scatterer if it fulfils the conditions in Definitions~\ref{defin:admissible class} and \ref{defin:admissible well-possed and non-vainishing}. The parameters $\{R,\Ncal, l_0,\epsilon_0,\alpha_m,
 \alpha_M\}$ in the above admissibility definitions are referred to as the a-priori parameters. It is remarked that the admissibility requirements in Definitions~\ref{defin:admissible class} and \ref{defin:admissible well-possed and non-vainishing} will be needed in our stability study and can be fulfilled in certain general scenarios, which shall become clearer in our subsequent discussion.

 Next, we present the logarithmic stability estimate for the inverse problem \eqref{eq:ip1} on the shape determination of the elastic medium scatterer by a single far-field pattern.
Henceforth, we define the Hausdorff distance of two medium {scatterers} $(\Omega;\rho)$ and $(\Omega';\rho')$ as follows,
\begin{equation}\label{eq:Hausdorff distance}
d_{\Hcal}(\Omega,\Omega')= \max\left\{ \sup_{\bx\in \Omega} \mathrm {dist}(\bx,\, \Omega'),\,  \sup_{\bx\in \Omega'} \mathrm {dist}(\bx,\, \Omega)\right\}.
\end{equation}

\begin{thm}\label{th:main1}
Let $(\Omega; \rho)$ and $(\Omega'; \rho')$ be two admissible polyhedral scatterers. Let $\bu^i$ be a common time-harmonic plane incident wave of the form \eqref{eq:incident wave}. Assume that $\bU$ and $\bU'$ are the far-field patterns of the scattered waves $\bu^s$ and $\bu'^s$ by the medium scatterers $(\Omega;\rho)$ and $(\Omega';\rho')$, respectively. For sufficiently small $\varepsilon\in \mathbb{R}_{+}$, if
$$
\|\bU-\bU'\|_{L^2(\mathbb{S}^{n-1},\mathbb{C}^n\times \mathbb{C}^n)}\leq \varepsilon,
$$
then
$$
d_{\Hcal}(\Omega,\Omega')\leq  C(\ln\ln(\Ncal/\varepsilon))^{-\gamma},
$$
where $\Ncal$ is given in Definition \ref{defin:admissible well-possed and non-vainishing}, and $C$ and $\gamma$ are positive constants, depending only on the a-priori parameters involved in Definitions \ref{defin:admissible class} and \ref{defin:admissible well-possed and non-vainishing} as well as the Lam\'{e} constants $\lambda, \mu$.
\end{thm}

The argument in proving Theorem~\ref{th:main1} is localized around a corner of the underlying polyhedral scatterer. As an interesting byproduct, we can establish another stability result where the medium scatterer is not necessarily polyhedral as long as it possesses a corner. The full technical details of the result will be given in Theorems~\ref{th:main2} and \ref{th:main2'}, and we only provide a rough summary in the following theorem.

\begin{thm}\label{thm:main2summary}
Consider the scattering problem \eqref{eq:incident wave}--\eqref{eq:rd1} associated with a general medium scatterer $(\Omega,\rho)$, which is not necessarily polyhedral. Suppose that $\partial\Omega$ possesses a corner as described in Theorem~\ref{th:main2'}. Then it holds that
\begin{equation}\label{eq:scattering bound1}
\|\bU\|_{L^2(\mathbb{S}^{n-1},\mathbb{C}^n\times \mathbb{C}^n)}\geq C,
\end{equation}
where $C$ is a positive constant depending on a certain set of a-priori parameters.
\end{thm}

Theorem~\ref{thm:main2summary} indicates that for a general medium scatterer, if it possesses a corner, then it generically scatters every incident field stably, i.e. invisibility phenomenon cannot occur.

\subsection{Connection to existing studies and discussion}

Determining the shape of an inhomogeneous object by minimal/optimal scattering measurements has been a longstanding problem in the literature with a long and colorful history; see \cite{CKreview,Liu21JIIP,LZ07} for reviews and surveys. Recently, several qualitative uniqueness results in determining convex polyhedral medium scatterers were established in \cite{LT20IP} for electrostatics, \cite{CX,HuandSalo2016,DCL21,CDL,BL20IP} for acoustic scattering, \cite{LX17,Blasten2019} for electromagnetic scattering and \cite{DLS} for elastic scattering. In \cite{BLSIMA21,LTY}, the shape determination of medium scatterers whose boundaries possess high-curvature parts was also considered for electrostatics and acoustic scattering. In \cite{Blasten2020,LT21}, quantitative stability estimates of double-logarithmic type were established in determining the convex polyhedral shape of an acoustic medium scatterer.

Another issue of significant physical interest is the occurrence of invisibility, namely the scattering pattern is identically zero. Generically, it is believed that geometric singularities on the support of a generic medium scatterer prevents the occurrence of the invisibility phenomenon. We refer to \cite{B,BLin,Blasten2014,Blasten2020,BLSIMA21,Blasten2019,CM,CX,DCL21,Liu21JIIP,LX17,Paivarinta2017,SS,VX} for related studies on this intriguing topic in different physical contexts. In \cite{BLSIMA21}, a geometrically singular point (say, e.g. a corner point) on the boundary of a shape is treated as with infinite curvature, and it is further shown that a medium scatterer whose smooth shape possesses a sufficiently high curvature point can also prevent the occurrence of invisibility. All of the aforementioned results are qualitative and in \cite{Blasten2020} sharp estimates were established by showing that there exist positive lower bounds of the scattering patterns, which quantify the non-invisibility phenomena due to the presence of the shape singularities. It is noted that the lower scattering bounds have been derived only for the acoustic scattering.

Our study in this article extends the related studies in \cite{Blasten2020,LT20IP,LT21} for the electrostatics and acoustic scattering to the elastic scattering which possesses more complicated and technical nature, both in physics and mathematics. It is pointed out that as remarked in \cite{Blasten2020}, the double-logarithmic stability estimates are generically optimal for inverse scattering problems. Finally, we would like to mention in passing a closely related topic on the geometric structures of transmission eigenfunctions which is beyond the scope of this article \cite{BL17,DCL21,DLS,CDHLW1,CDLS,DJLZ,DLWY,LT1,DLWW}.

The rest of the paper is organized as follows. In Section \ref{Auxiliary results}, we derive a critical auxiliary result about the propagation of smallness from far field to the boundary of the scatterer. Section \ref{sub:auxiliary3} is devoted to a micro-local analysis of the scattering solution around a corner. The full details of the proof of Theorem \ref{th:main1} is presented in Section \ref{sec:proTh1}. In Section~\ref{sec4}, we present the full technical details of Theorem~\ref{thm:main2summary}.

\section{Propagation of smallness from far-field to boundary}\label{Auxiliary results}
The main goal of this section is to show how the smallness from the far-field pattern propagates to the boundary of the elastic medium scatterer, which is a key ingredient in the stability proof of Theorem \ref{th:main1}. Our argument follows the general strategy developed in \cite{Blasten2020} for the acoustic scattering. Throughout the rest of the paper, we let $\bu$ and $\bu'$, respectively, denote the total wave fields corresponding to the scatterers $(\Omega; \rho)$ and $(\Omega'; \rho')$ in Theorem~\ref{th:main1}.

\subsection{Stability Estimates: from Far-field to Near-field}\label{sub:auxiliary1}
In this subsection, our aim is to estimate the difference of $\bu$ and $\bu'$ in $B_{2R}\backslash B_R$ close to the convex hull of $\Omega$ and $\Omega'$.
To begin with, we need to generalise Theorem 4.1 in \cite{Rondi2015} and Proposition 5.2 in \cite{Blasten2020} to the elastic inhomogeneous medium scatterer case.
\begin{lem}\label{lem:Isakov}
Fix $t\in (1,2]$ and $\Tcal>0$. Let $\bu\in H^2_{loc}(\mathbb{R}^n)^n$ be a solution to the Navier equation \eqref{eq:navier equation} in $\mathbb{R}^n\backslash \overline{B_R}$ with $\rho=1$, which satisfies the Kupradze radiation condition. Let $\bU$ be the far field pattern of $\bu$ with the norm $\varepsilon=\|\bU\|_{L^2(\mathbb{S}^{n-1},\mathbb{C}^n\times \mathbb{C}^n)}$. Assume that $\|\bu_\mathrm \alpha\|_{L^2(B_{2R}\backslash B_R)^n}\leq\Tcal, \mathrm\alpha =\mathrm p,\mathrm s$, where $\bu_\mathrm{p}$ and $\bu_\mathrm{s}$ are the longitudinal and the transversal parts of $\bu$, respectively. Then there exists a constant $C=C(\kappa_\mathrm p,\kappa_\mathrm s,\Tcal, n, R, t)>0$ such that the following stability estimate holds for $\varepsilon\leq\frac{\Tcal}{C}$:
\begin{equation}\label{eq:Isakov}
\|\bu\|_{L^2(B_{2tR}\backslash B_{tR})^n}\leq Ce^{-\hat{c}\sqrt{\ln\frac{\Tcal}{\varepsilon}}},
\end{equation}
where $\kappa=\min\{\kappa_\mathrm p,\kappa_\mathrm s\}$ and $\hat{c}\leq\ln t\sqrt{\frac{e\kappa R}{2}}$.
\end{lem}
\begin{proof}
Let $u^j$ and $u^j_\mathrm \alpha$ denote the $j$-th components of $\bu$ and $\bu_\mathrm \alpha$, respectively,  where \\$\mathrm \alpha=\mathrm p,\mathrm s$. From \eqref{eq:helmoltz decomposition1-1} and \eqref{eq:helmoltz decomposition1-2}, we obtain that $u^j_\mathrm \alpha$ satisfies the Helmoltz equation
$$
(\Delta+\kappa^2_\mathrm \alpha)u^j_\mathrm \alpha= 0,\quad j=1,2,\ldots,n.
$$
Then by  \cite[Proposition 5.2]{Blasten2020}, we know that if $\varepsilon\leq\frac{\Tcal}{C_{\mathrm \alpha}}$, then
\begin{equation*}
\|u^j_\mathrm \alpha\|_{L^2(B_{2tR}\backslash B_{tR})}\leq C_{\mathrm \alpha}\Tcal e^{-\hat{c}\sqrt{\ln\frac{\Tcal}{\varepsilon}}},
\end{equation*}
where $C_{\mathrm \alpha}=C_{\mathrm \alpha}(\kappa_{\mathrm \alpha},n,R,t)>0$, $\mathrm \alpha=\mathrm p,\mathrm s$, $\hat{c}\leq\ln t\sqrt{\frac{e\kappa R}{2}}$ and  $\kappa=\min\{\kappa_\mathrm p,\kappa_\mathrm s\}$.

When $\varepsilon\leq\frac{\Tcal}{\sqrt{2n\,(C^2_\mathrm p+C^2_\mathrm s)}}$, we can derive
\begin{align*}
\|\bu\|_{L^2(B_{2tR}\backslash B_{tR})^n}&=\sqrt{\sum_{j=1}^n\|u^j\|^2_{L^2(B_{2tR}\backslash B_{tR})}}\leq \sqrt{\sum_{\mathrm \alpha=\mathrm p,\mathrm s}\sum_{j=1}^n\|u^j_\mathrm \alpha\|^2_{L^2(B_{2tR}\backslash B_{tR})}}\\
&\leq \sqrt{2n\,(C^2_\mathrm p+C^2_\mathrm s)}\,e^{-\hat{c}\sqrt{\ln\frac{\Tcal}{\varepsilon}}}.
\end{align*}

The proof is complete.
\end{proof}

With the help of elliptic interior regularity, we can derive the following stability estimate of the near-fields in $ B_{2R}\backslash\overline{ B_R}$.

\begin{prop}\label{pro:far-field to near field}
Fix $t\in (1,2]$ and $\Tcal>0$. Assume that $\bu\in H^2_{loc}(\mathbb{R}^n)^n$ satisfies  the Navier equation $(\Delta ^*+\omega^2)\bu=\bf 0$ in $\mathbb{R}^n\backslash \overline{B_R}$ and the Kupradze radiation condition. Let $\bU$ be the far field pattern of $\bu$ with the norm $\varepsilon=\|\bU\|_{L^2(\mathbb{S}^{n-1},\mathbb{C}^n\times \mathbb{C}^n)}$. Assume that $\|\bu_\mathrm \alpha\|_{L^2(B_{2R}\backslash B_R)^n}\leq\Tcal, \mathrm\alpha =\mathrm p,\mathrm s$, where $\bu_\mathrm{p}$ and $\bu_\mathrm{s}$ are the longitudinal and the transversal parts of $\bu$, respectively. Let $K \subset B_{2R}\backslash \overline{B_{tR}}$ be a domain. Then there exists a positive constant $C=C(\kappa_\mathrm p,\kappa_\mathrm s,n, R, t)$ such that the following stability estimate holds for $\varepsilon\leq \frac{\Tcal}{C}$:
\begin{align}\label{pro:near-field estimation}
\|\bu\|_{H^r(K)^n}\leq  C \Tcal e^{-\hat{c}\sqrt{\ln\frac{\Tcal}{\varepsilon}}},
\end{align}
where $r\in \mathbb{N}$ is a smoothness index, $\kappa=\min\{\kappa_\mathrm p,\kappa_\mathrm s\}$ and $\hat{c}\leq\ln t\sqrt{\frac{e\kappa R}{2}}$.
\end{prop}

\begin{proof}
 Note that $\bu $ is the superposition of $\bu_{\mathrm p}$ and $\bu_{\mathrm s}$,
which are vector-valued weak solutions to the Helmholtz equation with wave numbers $\kappa_{\mathrm p}$ and $\kappa_{\mathrm s}$, respectively.
We use $u^j_\mathrm \alpha$ to denote the $j$-th component of $\bu_\mathrm \alpha$, where $\mathrm \alpha=\mathrm p,\mathrm s$.
Given two domains $\tilde{K}$ and $K'$ with $\tilde{K}\subset K'\subset B_{2R}\backslash \overline{B_{tR}}$ and $\mathrm {dist}(\partial\tilde{K},\partial K')>0$. For any $s\in \mathbb{R}$ and the smooth cutoff function $\tilde{\phi}\in C^{\infty}_0(K)$ with $\tilde{\phi}\equiv1$ in $\tilde{K}$,
it is not difficult to prove the two subsequent  properties:
\begin{equation*}
\|u^j_\mathrm \alpha\|_{H^{s+2}(\mathbb{ R}^n)}=\|(1+\kappa^2_\mathrm \alpha)u^j_{\mathrm \alpha}-(\Delta +\kappa^2_\mathrm \alpha)u^j_{\mathrm \alpha}\|_{H^{s}(\mathbb{ R}^n)}
\end{equation*}
and
\begin{equation*}
\|(\Delta +\kappa^2_\mathrm \alpha)(\tilde{\phi} \,u^j_{\mathrm \alpha})\|_{H^{s+2}(\mathbb{ R}^n)}=\|2\,\nabla\tilde{\phi}\cdot \nabla u^j_\mathrm\alpha+u^j_\mathrm \alpha\Delta\phi\|_{H^{s}(\mathbb{ R}^n)}\leq C_\mathrm\alpha\|u^j_\mathrm \alpha\|_{H^{s+1}(\Omega)}.
\end{equation*}
Thus we obtain
\begin{align*}
\|\bu\|^2_{H^{s+2}(\tilde{K})^n}&=\|\bu_\mathrm p+\bu_\mathrm s\|^2_{H^{s+2}(\tilde{K})^n}\leq\sum_{\mathrm \alpha=\mathrm p,\mathrm s}\sum^n_{j=1}\| u^j_\mathrm \alpha\|^2_{H^{s+2}(\tilde{K})}\\
&\leq 2\sum_{\mathrm \alpha=\mathrm p,\mathrm s}\sum^n_{j=1}\|\tilde{\phi}\, u^j_\mathrm \alpha\|^2_{H^{s+2}(\mathbb{R}^n)}\leq \sum_{\mathrm \alpha=\mathrm p,\mathrm s}\sum^n_{j=1}C^2_{\kappa_{\mathrm \alpha},\tilde{\phi}}\,\| u^j_\mathrm \alpha\|^2_{H^{s}(K')}\\
&\leq C_{\kappa_{\mathrm p},\kappa_{\mathrm s},n,\tilde{\phi}}^2\,\|\bu\|^2_{H^{s+1}(K')^n},
\end{align*}
which implies that 
\begin{equation}\label{eq:sobolve }
\|\bu\|_{H^{s+2}(\tilde{K})^n}\leq C_{\kappa_{\mathrm p},\kappa_{\mathrm p},n,\tilde{\phi}}\|\bu\|_{H^{s+1}(K')^n}.
\end{equation}
By fixing  $r\in\mathbb{N}$, there exists a subdomain sequence $\{ K_j,\phi_j\}^r_{j=0}$ such that
\begin{equation*}
\begin{cases}
&K_j\subset K_{j-1}, \,\, \mathrm{dist}(\partial K_j,\partial K_{j-1})>0,\\
&\phi_j\in C_0^\infty(K_{j-1}), \quad \phi_j\equiv 1 \quad \mathrm {in}\,\, K_j,\\
& K_0=B_{2R}\backslash \overline{B_{tR}},\quad\, K_r=K.
\end{cases}
\end{equation*}
By using \eqref{eq:sobolve } and Lemma \ref{lem:Isakov} repeatedly, we have
\begin{align*}
\|\bu\|_{H^{r}(K)^n}&\leq C_{\kappa _\mathrm p,\kappa_\mathrm s, \phi_r}\|\bu\|_{H^{r-1}(K)^n}\leq\cdots\leq C_{\kappa _\mathrm p,\kappa_\mathrm s, \phi_r,\ldots,\phi_1}\|\bu\|_{L^2(K)^n}\\
&\leq  C_{\kappa _\mathrm p,\kappa_\mathrm s, \phi_r,\ldots,\phi_1}\Tcal e^{-\hat{c}\sqrt{\ln(\Tcal/\varepsilon)}}.
\end{align*}

The proof is complete.
\end{proof}
\subsection{Stability Estimates: from Near-field to Boundary}\label{sub:auxiliary2}
In this subsection, we establish the propagation of smallness from the near-field to the boundary of $Q$ which is the convex hull of $\Omega$ and $\Omega'$. We are mainly interested in estimating the sum of $|\bu-\bu'|$ and $|\nabla\bu-\nabla\bu'|$ on $\partial Q$. To begin with, we introduce the following {\it three-spheres inequalities} which may be found, for instance, in \cite{Brummelhuis1995}.

\begin{lem}\label{lem:three-sphere inequality}\cite{Rondi2008}
There exist positive constants $\hat{R}$, $C$ and $c$, $0<c<1$, depending on $\kappa$ only. Let $0<r_1<r_2<r_3< \hat{R}$, and $u$ be a solution to $\Delta u+\kappa^2 u=0 $ in $B_{r_3}$. For any $s\in(r_2,r_3)$,  we have
$$
\|u\|_{L^\infty (B_{r_2})}\leq C\big(1-\frac{r_2}{s}\big)^{-3/2}\|u\|^{1-\beta}_{L^\infty (B_{r_3})}\|u\|^\beta_{L^\infty (B_{r_1})},
$$
where $\beta$ satisfies the following inequality,
\begin{equation*}
\frac{c\log\frac{r_3}{s}}{\log\frac{r_3}{r_1}}\leq \beta \leq 1-\frac{c\log\frac{s}{r_1}}{\log\frac{r_3}{r_1}}.
\end{equation*}
\end{lem}

From now on, we fix $r\in \mathbb{R}_+$ and set $r_1=r$, $r_2=2r$, $r_3=4r$, $s=2\sqrt{2}r$.

\begin{lem}\label{lem:curve three-sphere inequality}
Let $V\subset \mathbb{R}^n$ be a bounded connected domain and $\Gamma\subset V$ be a rectifiable curve with endpoints $\bx$ and $\bx'$. Fix positive constants $\Tcal$ and $\hat{R}$ such that $4r<\hat{R}$ and $\Tcal \geq 1$. Let $B(\Gamma, 4r)=\bigcup_{\by\in \Gamma}B_{4r}(\by)\subset V$. Moreover, we assume that $\bu\in L^\infty(V)^n$ satisfies
\begin{equation*}
\begin{cases}
&(\Delta^*+\omega^2)\bu={\bf 0},\quad \|\bu_\mathrm \alpha\|_{L^\infty(V)^n}\leq \Tcal,\, \mathrm \alpha=\mathrm p,\mathrm s,\\
&\max\big\{\|\bu_\mathrm p\|_{L^\infty(B_r(\bx))^n},\|\bu_\mathrm s\|_{L^\infty(B_r(\bx))^n}\big\}\leq 1,
\end{cases}
\end{equation*}
where $\bu_\mathrm{p}$ and $\bu_\mathrm{s}$ represent the longitudinal and the transversal parts of $\bu$, respectively. Then the following result holds
\begin{align*}
\|\bu_\mathrm \alpha\|_{L^\infty(B_r(\bx'))^n}\leq C\Tcal \|\bu_\mathrm \alpha\|^{\beta^{\frac{d_{\Gamma}}{r}+1}}_{L^\infty(B_r(\by))^n},\quad \mathrm \alpha=\mathrm p,\mathrm s,
\end{align*}
where $d_\Gamma$ is the distance measured along $\Gamma$.
\end{lem}
\begin{proof}
Let $N=\lceil d_\Gamma/r\rceil$. There exists a sequence of balls, each of radius $r$ and centred respectively at $\bx=\bx_1,\bx_2,\cdots,\bx_{N+1}=\bx'$, such that  $\bx_k\in\Gamma$ and $|\bx_{k+1}-\bx_k|\leq d_\Gamma(\bx_{k+1},\bx_k)\leq r$, where $d_\Gamma(\bx_{k+1},\bx_k)$ signifies the length along the curve $\Gamma$ between the points $\bx_k$ and $\bx_{k+1}$. Clearly, $B_r(\bx_k)\subset B_{2r}(\bx_{k-1})$. Let $u^j_\mathrm \alpha$ denote the $j$-th component of $\bu_\mathrm \alpha$, $\mathrm \alpha=\mathrm p,\mathrm s$. Then we have
\begin{align*}
\|\bu_\mathrm \alpha\|_{L^\infty(B_r(\bx'))^n}&\leq\sum^n_{j=1}\|u^j_\mathrm \alpha\|_{L^\infty(B_r(\bx'))}\leq \sum^n_{j=1}\Big\{C_1\Tcal^{1-\beta}\|u^j_\mathrm \alpha\|_{L^\infty(B_r(\bx_{N})) }  \Big\}\\
&\leq\cdots \\
&\leq \sum^n_{j=1}\Big\{C_{N}\Tcal^{(1-\beta)(1+\beta+\cdots+\beta^{N-1})}\|u^j_\mathrm \alpha\|^{\beta^{N}}_{L^\infty(B_r(\bx)) }  \Big\}.
\end{align*}
Note that $1+\beta+\cdots+\beta^N\leq 1/(1-\beta)$ and $N\leq d_\Gamma/r+1$, hence we have the following claim
\begin{align*}
\|\bu_\mathrm \alpha\|_{L^\infty(B_r(\bx'))^n}\leq C\Tcal\|\bu_\mathrm \alpha\|^{\beta^{N}}_{L^\infty(B_r(\bx))^n}\leq C\Tcal\|\bu_\mathrm \alpha\|^{\beta^{\frac{d_\gamma}{r}+1}}_{L^\infty(B_r(\bx))^n},\quad \mathrm \alpha=\mathrm p,\mathrm s.
\end{align*}

The proof is complete.
\end{proof}

We now give the stability estimations of $\bu_\mathrm p$ and $\bu_\mathrm s$ near the scatterer $(\Omega;\rho)$.
\begin{prop}\label{pro:auxilary1}
  Given positive constants $\Tcal,\hat{R}$ and $\theta,\beta\in (0,1)$ such that $\Tcal \geq1$.
  Let $Q\subset B_R$ be a convex polytope. Let $\bu\in H^2_{loc}(B_{2R})^n$ satisfy the Navier equation $\Delta^*\bu+\omega^2\bu=\bf 0$ in $B_{2R}\backslash \overline{Q}$. Moreover, we assume that
  $$
  \|\bu_\mathrm \alpha\|_{C^\theta\big(\overline{B_{3/2R}}\big)^n}\leq \Tcal,\quad \|\bu_\mathrm \alpha\|_{L^\infty(B_{7R/4}\backslash B_{5R/4})^n}\leq \delta < 1,\,\, \alpha=\mathrm p,\mathrm s,
  $$
  where $\bu_\mathrm{p}$ and $\bu_\mathrm{s}$ are the longitudinal and the transversal parts of $\bu$, respectively. Let $r=\frac{9R|\ln\beta|}{4(1-\theta)\ln|\ln\delta|}$. If
 \begin{equation}\label{ineq:upper bound}
 \delta<\Big[\exp\exp \Big\{\frac{4\,r\ln|\ln\beta|}{\min\{\hat{R},R/2\}}         \Big\}\Big]^{-1},
 \end{equation}
 then
 \begin{align}\label{ineq:estimate}
 \|\bu_\mathrm \alpha\|_{L^\infty(Q')^n}\leq M(\ln|\ln\delta|)^{-\theta}\Tcal,\quad \mathrm \alpha=\mathrm p,\mathrm s.
 \end{align}
Here $Q'=\big\{\bx\in B_{3/2R}\,\big|\, \mathrm{dist} (\bx,\partial Q)\leq 4\, r\big\}$ and $ M= M(\omega, R, \theta,\beta)>1$.
 \end{prop}
 \begin{proof}
Let $Q_1=\big\{\bx\in B_{2R}\,\big| \,\mathrm{dist} (\bx,\partial Q)< 4\,r\big\}$. It is easy to reduce that $4\,r<\hat{R}$ and $2\,r <\frac{R}{4}$ from the upper bound on $\delta$ and the expression of $r$. For any point $\bx'\in B_{2R}\backslash \overline{Q_1}$, there always exists a ray from $\bx'$ into $B_{2R}\backslash B_{5/4R}$.
We consider a special segment of the ray with endpoints $\bx'$ and $\by$, where $\by\in B_{7/4R}\backslash B_{5/4R}$ and $\mathrm{dist}(\by,\partial B_{5/4R})=r$. One can see that $\|\bu_{\mathrm\alpha}\|_{L^\infty(B_{r}(\by))^n}\leq \delta, \mathrm \alpha=\mathrm p,\mathrm s$. The length of that segment is smaller than $\frac{5R}{2}+r$. By Lemma \ref{lem:curve three-sphere inequality}, we have
\begin{align*}
\|\bu_\mathrm \alpha\|_{L^\infty(B_{r}(\bx'))^n}\leq C\Tcal \|\bu_\mathrm \alpha\|^{\beta^{\frac{5R}{2\,r}+2}}_{L^\infty(B_{r}(\by))^n}\leq C \Tcal\delta^{\beta^{\frac{5R}{2r}+2}},\quad \mathrm \alpha=\mathrm p,\mathrm s.
\end{align*}
Then
\begin{align*}
\|\bu_\mathrm \alpha\|_{L^\infty(B_{2R}\backslash Q_1)^n}\leq C' \Tcal\delta^{\beta^{\frac{5R}{2r}+2}},\quad \mathrm \alpha=\mathrm p,\mathrm s.
\end{align*}
For any point $\bx'\in Q_1$, there must be $\by\in\partial Q$ such that $|\bx'-\by|\leq 4\,r$. By the convexity of $Q$, there exists $\bx\in\mathbb{R}^n\backslash Q$ such that $\mathrm{dist}(\bx, Q)=|\bx-\by|=4\,r$. The upper bound of $\delta$ implies $4\,r\leq \frac{R}{2}$, and thus $$|\bx|\leq|\bx-\by|+|\by|\leq 4\,r+R\leq \frac{3R}{2}.$$
Moreover, $|\bx'-\bx|\leq |\bx'-\by|+|\by-\bx|\leq 8\,r$.
 From the $\theta$-H\"{o}lder continuity of $\bu_\mathrm\alpha$ and proposition 5.7 in \cite{Blasten2020}, we can obtain that
\begin{align}\label{ineq:holder}
\|\bu_\mathrm\alpha\|_{L^\infty(Q_1)^n}&\leq \sum^n_{j=1}\|u^j_\mathrm \alpha\|_{L^\infty(Q_1)}\leq\sum^n_{j=1} \Big(\|u^j_\mathrm\alpha\|_{\bC^\theta(B_{3/2R})} |\bx-\bx'|^\theta+\|u^j_\mathrm \alpha\|_{L^\infty(Q_1)}\Big)\nonumber\\
&\leq\sum^n_{j=1}C\Tcal\Big(  (8\,r)^{\theta}+\delta^{\beta^{\frac{5R}{2\,r}+2}}\Big)\leq\widetilde{C}\Tcal\Big(  (8\,r)^{\theta}+\delta^{\beta^{\frac{5R}{2\,r}+2}}\Big),
\end{align}
where $\widetilde{C}$ is a positive constant depending on $\kappa_\mathrm \alpha$ and $n$.

Combining 
\eqref{ineq:upper bound} with $|\ln\delta|>1$, we have
\begin{align}\label{eq:delta1}
r^\theta=\frac{(9R|\ln\beta|)^\theta}{((4-4\theta)\ln|\ln\delta|)^{\theta}},\quad \quad \frac{5R}{2\,r}=\frac{10(1-\theta)\ln|\ln\delta|}{9|\ln\beta|},
\end{align}
and
\begin{align}\label{eq:delta2}
\delta^{\beta^{\frac{5R}{2\,r}+2}}&=e^{-|\ln\delta|\beta^{\frac{5R}{2\,r}+2}}=e^{-\beta^2|\ln\delta|^{\frac{10\theta-1}{9}}}\leq e^{-\beta^2|\ln\delta|^\theta}\nonumber\\
&\leq\beta^{-2}(|\ln\delta|)^{-\theta} \,\leq \,\beta^{-2}(\ln|\ln\delta|)^{-\theta}.
\end{align}
Therefore, it  yields that
 \begin{align*}
 \|\bu_\mathrm \alpha\|_{L^\infty(Q')^n}\leq\|\bu_\mathrm \alpha\|_{L^\infty(Q_1)^n}\leq M(\ln|\ln\delta|)^{-\theta}\Tcal,\quad \mathrm \alpha=\mathrm p,\mathrm s,
 \end{align*}
 where $M=\widetilde{C}\big[2^\theta R^\theta|\ln\beta|^\theta(1-\theta)^{-\theta}+\beta^{-2}\big]$.

The proof is complete.
 \end{proof}

The next proposition states the propagation of smallness from the near-field to the boundary of $Q$, which is the convex hull of $\Omega$ and $\Omega'$.

\begin{prop}\label{pro:estimate on boundary}
Fix $n=2,3$ and the a-prior parameters $\Mcal,\Pcal$ and $\Ncal$ greater than $1$. Let $(\Omega;\rho)$ and $(\Omega'; \rho')$ be two medium scatterers associated with
two admissible density functions $\rho$ and $\rho'$. Let $\bu^i$ be an incident wave with the form \eqref{eq:incident wave} and $\|\bu^i\|_{H^2(B_{2R})^n}\leq\Pcal$. Assume that $\bU$ and $\bU'$ are the far field patterns of the scattered waves $\bu^s$ and $\bu'^s$ by the medium scatterers $(\Omega;\rho)$ and $(\Omega';\rho')$, respectively, where $\bu^s,\bu'^s$ can be bounded by $\Ncal>1$ in $H^2(B_{2R})^n$. Let $\varepsilon$ be a sufficiently small constant. If
$$
\|\bU-\bU'\|_{L^2(\mathbb{S}^{n-1},\mathbb{C}^n\times \mathbb{C}^n)}\leq \varepsilon,
$$
then $\bu-\bu'$ and $\nabla \bu-\nabla\bu'$ are continuous in $B_R$ and moreover, we have
\begin{align}\label{ineq:estimate+boundary}
\sup_{\partial Q}\big( |\bu-\bu'|+|\nabla \bu-\nabla \bu'|  \big)\leq C\left(\ln\ln (\Ncal/\varepsilon)\right)^{-1/2},
\end{align}
where $Q$ is the convex hull of $\Omega$ and $\Omega'$, and $C=C(\kappa_\mathrm p,\kappa_\mathrm s,n, R, \Pcal,\Mcal,\Ncal)>0$.
\end{prop}
\begin{proof}
Denote $\varepsilon_0=\|\bU-\bU'\|_{L^2(\mathbb{S}^{n-1},\mathbb{C}^n\times \mathbb{C}^n)}$ and $\bv=\bu-\bu'=\bu^s-\bu'^s$. One can easily see that $\|\bv\|_{H^2(B_{2R})}\leq 2\Pcal$. Let $\Omega=B_{7/4R}\backslash B_{5/4R}$. From Proposition \ref{pro:far-field to near field} and the Sobolev embedding $H^2(\Omega)\rightarrow L^{\infty}(\Omega)$, there exists a positive constant $C_1=C_1(\kappa_\mathrm p,\kappa_\mathrm s,n, R)$ such that the following estimates hold for sufficiently small $\varepsilon$,
\begin{align*}
&\|\bv\|_{L^\infty(\Omega)^n}\leq \|\bv\|_{H^2(\Omega)^n}\leq C_1 \Ncal e^{-c_0\sqrt{\ln\frac{\Ncal}{\varepsilon_0}}}\leq C_1 \Ncal e^{-c_0\sqrt{\ln\frac{\Ncal}{\varepsilon}}},\\
 &\|\nabla\bv\|_{L^\infty(\Omega)^n}\leq \|\bv\|_{H^2(\Omega)^n}\leq C_1\Ncal e^{-c_0\sqrt{\ln\frac{\Ncal}{\varepsilon_0}}}\leq  C_1\Ncal e^{-c_0\sqrt{\ln\frac{\Ncal}{\varepsilon}}},
\end{align*}
where $\kappa=\min\{\kappa_\mathrm p,\kappa_\mathrm s\}$ and $c_0=\ln (\frac{5}{4})\sqrt{\frac{e\kappa R}{2}}$. According to the interior  regularity of the elliptic PDE (cf. \cite{Gilbarg1983}) and the Sobolev embedding, it is not difficult to prove that
\begin{align*}
&\|\bu_\mathrm\alpha\|_{C^{1,\frac{1}{2}}(\overline {B_{3/2R}})^n}\leq C_\mathrm\alpha(1+\|\rho\|_{C^{\frac{1}{2}}}(B_{3/2R})^n)\|\bu_\mathrm\alpha\|_{H^2(B_{2R})^n}\leq C_\mathrm\alpha(1+\Mcal)(\Pcal+\Ncal),\\
&\|\bu'_\mathrm\alpha\|_{C^{1,\frac{1}{2}}(\overline {B_{3/2R}})^n}\leq C_\mathrm\alpha(1+\|\rho\|_{C^{\frac{1}{2}}}(B_{3/2R})^n)\|\bu'_\mathrm\alpha\|_{H^2(B_{2R})^n}\leq C_\mathrm\alpha(1+\Mcal)(\Pcal+\Ncal),
\end{align*}
where $(\bu_\mathrm s,\bu_\mathrm p )$ and $(\bu'_\mathrm s,\bu'_\mathrm p )$ are transverse and longitudinal wave pairs of $\bu$ and $\bu'$, respectively.
Thus $\bv$ and $\partial_j\bv \in C^{\frac{1}{2}}(\overline{B_{3/2R}})^n,\,j=1,2,\ldots,n$. This shows that $\bu-\bu'$ and $\nabla \bu-\nabla\bu'$ are continuous in $B_R$.

Fix $\beta\in(0,1)$ and denote $Q'=\big\{\bx\in B_{3/2R}\big| \,\,\mathrm{dist} (\bx,\partial Q)\leq 4\cdot\frac{9R|\ln\beta|}{2\ln|\ln\delta|}\big\}$. We choose $\delta(\varepsilon)=C_2\,\Ncal e^{-c_0 \sqrt{\ln(\Ncal/\varepsilon)}}$ satisfying \eqref{ineq:upper bound}. According to Proposition \ref{pro:auxilary1}, we know that
\begin{align*}
\|\bv_\mathrm \alpha\|_{L^\infty(Q')^n}&\leq C_\mathrm\alpha(1+\Mcal)(2\Pcal+\Ncal)(\ln|\ln\delta|)^{-\frac{1}{2}},\quad \mathrm \alpha=\mathrm p,\mathrm s.
\end{align*}
Hence, one can deduce that
\begin{align}\label{eq:v bound}
\|\bv\|_{L^\infty(Q')^n}&\leq \max\{C_\mathrm p,C_\mathrm s\}(1+\Mcal)(2\Pcal+\Ncal)(\ln\ln|\delta|)^{-\frac{1}{2}}\notag \\
                        &\leq \max\{C_\mathrm p,C_\mathrm s\}(1+\Mcal)(2\Pcal+\Ncal)\Big[\ln\big( \frac{c_0}{2}\ln(\Ncal/\varepsilon)\,\big)\Big]^{-\frac{1}{2}}\notag \\
                        &\leq \max\{C_\mathrm p,C_\mathrm s\}(1+\Mcal)(2\Pcal+\Ncal)\big(\ln \ln(\Ncal/\varepsilon)^{\frac{1}{4}}\big)^{-\frac{1}{2}}\notag \\
                        &\leq2 \max\{C_\mathrm p,C_\mathrm s\}(1+\Mcal)(2\Pcal+\Ncal)\big(\ln \ln(\Ncal/\varepsilon)\big)^{-\frac{1}{2}}.
\end{align}
Similar to the upper bound  of $\|\bv\|_{L^\infty(Q')^n}$ in \eqref{eq:v bound}, we can obtain the upper bound of $\partial_j\bv$ by following a similar argument. Hence, we can prove \eqref{ineq:estimate+boundary}.

The proof is complete.
\end{proof}

\section{Micro-local analysis of corner scattering}\label{sub:auxiliary3}

The present section is devoted to analyzing the quantitative behaviours of the scattered field locally around a corner. We first present two auxiliary lemmas.


\begin{lem}\label{lem:convex hull_2D} \cite[Lemma 8.2]{Blasten2020}
Fix $n=2$. Let $Q$ be the convex hull of $\Omega$ and $\Omega'$, where $\Omega$ and $\Omega'$ are two open bounded convex polygons. If $\bx_0$ is a vertex of $\Omega$ such that $\mathrm{dist}(\bx_0,\Omega')=\mathrm{dist}_{\Hcal}(\Omega,\Omega')$, then $\bx_0$ is also a vertex of $Q$. Let $\sigma$ denote the angle of $K$ at $\bx_0$, then the angle $\sigma'$ of $Q$ at $\bx_0$ satisfies the inequality $ \sigma<\sigma'\leq (\sigma+\pi)/2<\pi$.
\end{lem}

\begin{lem}\label{lem:convex hull 3D} \cite[Lemma 8.3]{Blasten2020}
Fix $n=3$. Let $Q$ be the convex hull of $\Omega$ and $\Omega'$, where $\Omega$ and $\Omega'$ are open convex polyhedral cones. If $\bx_0$ is a vertex of $\Omega$ such that $\mathrm{dist}(\bx_0,\Omega')=\mathrm{dist}_{\Hcal}(\Omega,\Omega')$, then $\bx_0$ is also a vertex of $Q$. And $Q$ can fit inside such an open spherical cone $\Acal$ that $\bx_0$ is the vertex of $\Acal$ whose open angle $\sigma$ at the vertex $\bx_0$ is smaller than $\pi$, where $\sigma$ does not depend on $\Omega$, $\Omega'$ and their locations.
\end{lem}

Next, we establish an integral identity which follows from the Betti's second formula.

\begin{prop}\label{prop:intergal identity}
Let $\omega>0$ and $\rho(\bx)\in C^\theta(S) $, where $S\subset \mathbb{R}^n$ is a bounded Lipschitz domain and $\theta\in(0,1)$.
 If $\bu,\bu'$ and $\bu_0\in H^2(S)^n$ satisfy the Navier equations
\begin{align*}
\left\{ \begin{array}{ll}
\Delta^*\bu\,+\rho\,\omega^2\bu=\bf 0,\\
\Delta^*\bu'+~\omega^2\bu'=\bf 0,\\
\Delta^*\bu_0+\omega^2\bu_0=\bf 0
\end{array}
 \right.
\end{align*}
in $S$, then
\begin{align}\label{eq:intergal identity}
\omega^2\int_{S}(\rho-1)\,\bu\cdot \bu_0 \,\mathrm{d}\bx=\int_{\partial S }\bu_0\cdot T_{\hat{\nu}}(\bu^i-\bu)-(\bu^i-\bu)\cdot T_{\hat{\nu}}\,(\bu_0 )\,\mathrm{d \sigma}.
\end{align}
Here, $\hat{\nu}$ denotes the outward unit normal to the boundary of $S$ and the the conormal derivative $T_{\hat{\nu}}(\bu)$ with $u^j$ being the $j$-th component of $\bu$ is defined by
\begin{align}\label{eq:surface traction}
T_{\hat{\nu}}(\bu)&=
\left\{\begin{array}{lr}
2\mu\,\partial _{\hat{\nu}}\bu+\lambda\,\hat{\nu}(\nabla\cdot\bu)+\mu\,(\partial_2 u^1-\partial_1u^2)\hat{\nu}^\perp,&\quad n=2, \\[10pt]
2\mu\,\partial_{\hat{\nu}}\bu+\lambda\,\hat{\nu}\,(\nabla\cdot \bu)+\mu\,\hat{\nu}\times (\nabla\times \bu),&\quad n=3.
\end{array}\right.
\end{align}
\end{prop}

The following complex geometric optics (CGO) solution $\bu_0$ is introduced in \cite{Haher1998}:
\begin{equation}\label{eq:cgo}
\bu_0(\bx)=e^{\xi\cdot(\bx-\bx_0)}\eta,
\end{equation}
where
\begin{align}\label{eq:cgo1-2}
&\xi=\tau\bp+\imath\sqrt{\kappa^2_\mathrm s+\tau^2}\,\bp^\perp,\, \eta=\bp^\perp-\imath\sqrt{1+\kappa^2_\mathrm s/\tau^2}\,\bp,\nonumber\\
 &\tau>\kappa_\mathrm s,\quad \bp\cdot \bp^\perp=0,\quad \bp^\perp,\, \bp\in \mathbb{S}^{n-1}.
\end{align}
 It is directly verified that
\begin{equation}\label{eq:cgo equality}
	\Delta^* \bu_0+ \,\omega^2 \bu_0={\bf 0}\ \mbox{ in }\ \mathbb R^n.
\end{equation}
Note that
\begin{align}\label{eq:cgo+property}
|\xi|=\sqrt{2\tau^2+\kappa^2_\mathrm s}\quad\mbox{and}\quad |\eta|=\sqrt{2+\kappa^2_\mathrm s/\tau^2}\leq\sqrt{3}.
\end{align}

In Proposition \ref{prop:cgo} we shall present asymptotic
   analysis of the volume integral of the CGO solution $\bu_0$ over a polyhedral cone with respect to $\tau$ as $\tau \rightarrow +\infty$. Before proving that, we now introduce a special cone pair $(\Acal,\Ccal)$ as follows.
\begin{defn}\label{def:cone_pair}
Let $n\in\{2,3\}$, $0<\alpha_m<\alpha_M<\frac{\pi}{2}$, $\omega>0$. Assume that $\Acal,\Ccal\subset\mathbb{R}^n$ are, respectively, an open polyhedral cone and an open spherical cone. We say $(\Acal,\Ccal)\in\Fcal(\alpha_m,\alpha_M,n)$ if the following conditions are fulfilled:
\begin{itemize}
\item[{\rm (a)}] $\Acal\subset\mathbb{R}^n$ is a convex polyhedral cone;
\item[{\rm (b)}] $\Acal$, $\Ccal$ both take $\bx_0\in\mathbb{R}^n$ as the apex and $\Acal\subset\Ccal$;
\item[{\rm (c)}] the opening angle of $\Ccal$ at the apex $\bx_0$ is no more than $\pi$;
\item[{\rm (d)}] the opening angle of $\Acal\subset\mathbb{R}^2$ at  the apex $\bx_0$ is in $(2\alpha_m, 2\alpha_M)$.
\end{itemize}
\end{defn}

\begin{prop} \label{prop:cgo}
Fix $n\in\{2,3\}$, $0<2\alpha_m<2\alpha_M<\pi$, $\omega>0$. Let $\bu$ be the solution to \eqref{eq:navier equation} with $\Omega=\mathrm {supp}(\rho-1)$. In addition, $\Omega$ takes $\bx_0$ as its apex and $\bu(\bx_0)$ is a nontrivial constant complex vector. Let $\Acal$ and $\Ccal$ be an open polyhedral cone and an open spherical cone, respectively, they both take $\bx_0$ as their vertex. If $(\Acal,\Ccal)\in\Fcal(\alpha_m,\alpha_M,n)$, then there exists a vector $\xi$ defined by \eqref{eq:cgo1-2} 
  such that
\begin{align}\label{ineq:delta_0}
&\bp \cdot(\bx-\bx_0)\leq -\delta_0\big|\bx-\bx_0\big| \mbox{ for any }\,  \bx \in  \Acal,
\end{align}
where $\delta_0$ is a positive constant only depending on $\bp$, $\Omega$ and $\Acal$.
Moreover, we have
\begin{align}\label{eq:lower bound of CGO}
\Big| \int_{\Acal} e^{\xi\cdot(\bx-\bx_0)}\mathrm dx        \Big|\geq C\,\tau^{-n},
\end{align}
and
\small{\begin{align}\label{eq:lower bound of CGO2}
\big|\bu(\bx_0)\cdot \eta\big|\geq C_0>0
,
\end{align}}where $\tau=\big|\Re\xi\big|$, $\eta$ is defined in \eqref{eq:cgo1-2}, $\bp^\perp\in \mathbb{S}^{n-1}$ satisfies $\bp\cdot \bp^\perp=0$, and $C_0$ is a positive constant only depending on $\Re\bu(\bx_0)$, $\Im\bu(\bx_0)$, $\bp$ and $\bp^\perp$. Here $\Re\bu(\bx_0)$ and $\Im\bu(\bx_0)$ stand for the real and imaginary parts of $\bu(\bx_0)$, respectively.
\end{prop}
\begin{proof}
 The above conclusions can be obtained directly from \cite[Lemma 6.3]{Blasten2020} and \cite[Lemma 2.2]{Blasten2019} except that we need to prove the inequality \eqref{eq:lower bound of CGO2}.
Note that
\small{\begin{align}\label{eq:u(x0)}
\bu(\bx_0)\cdot\eta=&\,\,\imath\big(\Im\bu(\bx_0)\cdot \bp^\perp -\sqrt{1+\kappa^2_\mathrm s/\tau^2}\,\Re\bu(\bx_0)\cdot\bp\big)\nonumber\\
&+ \Re\bu(\bx_0)\cdot \bp^\perp+\sqrt{1+\kappa^2_\mathrm s/\tau^2}\,\Im\bu(\bx_0)\cdot \bp.
\end{align}}Since $\bu(\bx_0)$ is a nontrivial constant complex vector, without loss of generality, we assume that $\Im\bu(\bx_0)\neq\bf0$. For simplicity, we only consider three-dimensional scenarios and the two-dimensional conclusion can be similarly proved. In fact, once $\bp$ is fixed, $\bp^\perp$ belongs to $\Pi$, where $\Pi$ represents the  plane defined by $ \Pi=\{\mathbf x \in \mathbb R^3~|~\mathbf x \cdot \bp =0\}$. We can easily prove \eqref{eq:lower bound of CGO2} by virtue of choosing a special unit  vector $\bp^\perp$ as follows:
\begin{itemize}
	\item[{\rm (I)}] Case 1: $\Im\bu(\bx_0)\cdot \bp< 0$. In order to prove \eqref{eq:lower bound of CGO2}, we consider the lower bound for the real part of $\mathbf u(\mathbf x_0) \cdot \eta.$

Let us distinguish two seperate situations.  If $\Re\bu(\bx_0)\in \Pi$ (including $\Re\bu(\bx_0)\px \Pi$), we can take $\bp^\perp \in \Pi$ such that $\angle( \bp^\perp, \Re\bu(\bx_0) ) \in (\pi/2,\pi) $, where $\angle( \bp^\perp, \Re\bu(\bx_0) )$ is the angle between $\Re\bu(\bx_0)$ and $\bp^\perp$.  If $\Re\bu(\bx_0)\notin \Pi$, it is clear that there must exist a unit  vector $\bp^\perp \in \Pi$ such that $\angle( \bp^\perp, \Re\bu(\bx_0) ) \in (\pi/2,\pi). $ Hence, we always can choose the unit vector $\bp^\perp \in \Pi$ fulfilling the following inequality,
\begin{equation}\notag 
\Re\bu(\bx_0)\cdot \bp^\perp< 0.
\end{equation}
It  can be directly seen that the signs of $\Re\bu(\bx_0)\cdot \bp^\perp$ and $\Im\bu(\bx_0)\cdot \bp $ are the same.
Therefore, we have
\begin{align}\label{ineq:u(x_0)}
\big|\bu(\bx_0)\cdot \eta\big|&\geq\big| \Re\bu(\bx_0)\cdot \bp^\perp+\sqrt{1+\kappa^2_\mathrm s/\tau^2}\,\Im\bu(\bx_0)\cdot \bp\big| \nonumber\\
&\geq\big| \Re\bu(\bx_0)\cdot \bp^\perp+\Im\bu(\bx_0)\cdot \bp\big | :=C_0>0.
\end{align}
\item[{\rm (II)}]  Case 2: $\Im\bu(\bx_0)\cdot \bp= 0$.  In order to prove \eqref{eq:lower bound of CGO2}, we consider the lower bound for the imaginary  part of $\mathbf u(\mathbf x_0) \cdot \eta.$

Let us investigate the following two situations. If $\Re\bu(\bx_0)\in \Pi$ (including $\Re\bu(\bx_0)\px \Pi$), then $\Re\bu(\bx_0)\cdot\bp=0$. Thus we can take $\bp^\perp \in \Pi $ such that  $\angle(\bp^\perp,  \Re\bu(\bx_0))\neq \pi/2$.  If $\Re\bu(\bx_0)\notin \Pi$, let us suppose that $\Re\bu(\bx_0)\cdot\bp>0$, we  can choose a unit vector $\bp^\perp$ belonging to $\Pi$ such that $\angle(\bp^\perp,  \Im\bu(\bx_0))\in (\pi/2,\pi)$, which implies ${\sf sign}(\Im\bu(\bx_0)\cdot \bp^\perp)=-{\sf sign}(\Re\bu(\bx_0)\cdot\bp)$. Here ${\sf sign}(a)$ is the sign of a real  number $a$.
Therefore, we obtain
\begin{align*}\label{eq:deter2}
\big|\bu(\bx_0)\cdot \eta\big|&\geq\big|\Im\bu(\bx_0)\cdot \bp^\perp -\sqrt{1+\kappa^2_\mathrm s/\tau^2}\,\Re\bu(\bx_0)\cdot\bp\big| \\
&\geq\big| \Im\bu(\bx_0)\cdot \bp^\perp -\Re\bu(\bx_0)\cdot\bp\big|:=C_0>0.
\end{align*}
\item[{\rm (III)}] Case 3: If $\Im\bu(\bx_0)\cdot \bp>0$. Similar to Case 1,  in order to prove \eqref{eq:lower bound of CGO2}, we can consider the lower bound for the real part of $\mathbf u(\mathbf x_0) \cdot \eta.$

Similar to the above two cases, let  us consider the following two situations. If $\Re\bu(\bx_0)\in \Pi$ (including $\Re\bu(\bx_0)\px \Pi$), we can take $\bp^\perp  \in \Pi $ such that $ \angle( \bp^\perp , \Re\bu(\bx_0) )\in  [0,\pi/2)$. If $\Re\bu(\bx_0)\notin \Pi$, there must exist a  unit  vector  $\bp^\perp \in \Pi$ such $\angle(\bp^\perp, \Re\bu(\bx_0))\in  [0, \pi/2)$. Hence, for the above two  situations,  we can always choose a unit vector $\bp^\perp$ fulfilling the following inequality as $\bp^\perp$,
\begin{equation}\notag 
\Re\bu(\bx_0)\cdot \bp^\perp\geq 0.
\end{equation}
By the same deduction for Case 1, we can  also obtain the inequality \eqref{ineq:u(x_0)}.
\end{itemize}

The proof is complete.
\end{proof}

A critical asymptotic estimate with respect to $\Re \xi $  for the volume integral associated with the CGO solution $\bu_0$ and the total wave field $\bu$  near a vertex of $\Omega$ is presented in the following proposition, which is the key ingredient to derive the stability result in Theorem \ref{th:main1}.

\begin{prop}\label{prop:important estimate}
Let $\bu$ and $\bu'$ be the solutions to \eqref{eq:navier equation} with $\Omega,\Omega', \rho, \rho'$ satisfying the assumptions in Theorem \ref{th:main1}. Let $\bx_0$ be a vertex of $\Omega$ such that $\mathrm{dist}(\bx_0,\Omega')=\mathrm{dist}_{\Hcal}(\Omega,\Omega')$. Fix $\omega,\tau>0$ and $h\in(0,1)$ such that $B_h(\bx_0)\cap \Omega'=\emptyset$. Denote $S_h=B_h(\bx_0)\cap \Omega$ and $D_h=B_h(\bx_0)\cap Q$, where $Q$ is the convex hull of $\Omega$ and $\Omega'$. Let $\bu(\bx_0)$ be a nontrivial constant complex vector. Assume that $\Acal$ and $\Bcal$ represent the polyhedral cone generated by $K$ and $Q$ at the vertex $\bx_0$, respectively.
Let $\rho \in C^{\theta_1}(S_h)$ and $\bu \in C^{\theta_2}(S_h)^n$, $\theta_j\in (0,1)$, $j=1,2$, which yield that
\begin{align}
&\rho(\bx)=1+\rho(\bx_0)+\rho_1(\bx),\,\,\,|\rho_1(\bx)|\leq \Mcal|\bx-\bx_0|^{\theta_1},\,\,\theta_1\in(0,1),\label{eq:rho expan}\\
&\bu(\bx)=\bu(\bx_0)+\bu_1(\bx),\quad\quad |\bu_1(\bx)|\leq \Rcal |\bx-\bx_0|^{\theta_2},\,\,\theta_2\in(0,1),\label{eq:u expan}
\end{align}
where $\Mcal =\|\rho_1\|_{C^{\theta_1}(\Omega)}$ and $\Rcal=\|\bu_1\|_{C^{\theta_2}(\Omega)}$.

Then we have the integral identity
\begin{align}\label{ineq:core identity}
(\rho(\bx_0)-1)\int_{\Acal} e^{\xi\cdot (\bx-\bx_0)}\eta\cdot \bu(\bx_0)\mathrm{d}x&=\frac{1}{\omega^2}\int_{\partial D_h}\bu_0\cdot T_{\hat{\nu}}(\bu^i-\bu)-(\bu^i-\bu)T_{\hat{\nu}}(\bu_0)\mathrm{d\sigma}\nonumber\\
&\quad+(\rho(\bx_0)-1)\int_{\Acal\backslash S_h} e^{\xi\cdot (\bx-\bx_0)}\eta\cdot \bu(\bx_0)\mathrm{d}x\nonumber\\
&\quad-\int_{S_h}e^{\xi\cdot (\bx-\bx_0)}(\rho_1+1)\eta\cdot \bu(\bx_0)\mathrm{d}x\nonumber\\
&\quad-\int_{S_h}e^{\xi\cdot (\bx-\bx_0)}(\rho-1)\eta\cdot \bu_1\mathrm{d}x,
\end{align}
and the estimate
\begin{align}\label{ineq:core estimate}
C\big|(\rho(\bx_0)-1)\int_{\Acal} e^{\xi\cdot (\bx-\bx_0)}\eta\cdot \bu(\bx_0)\mathrm{dx}\big|&\leq |\Re\xi|^{-n}e^{-\delta_0|\Re\xi|h/2}+|\Re\xi|^{-n-\min\{\theta_1,\theta_2\}}\nonumber\\
&+h^{n-1}\sup_{\partial\Bcal\cap B_h(\bx_0)}\big\{|\nabla \bu-\nabla \bu'|+|\bu-\bu'|  \big\}\nonumber\\
&+h^{n-1}e^{-\delta_0|\Re\xi|h}\big\{ \|\bu_\mathrm s\|_{H^2(B_{2R})^n}+\|\bu_\mathrm p\|_{H^2(B_{2R})^n}\nonumber\\
 &+\|\bu'_\mathrm p\|_{H^2(B_{2R})^n} +\|\bu'_\mathrm s\|_{H^2(B_{2R})^n}               \big\},
\end{align}
where $\delta_0$ coincides with the one in \eqref{ineq:delta_0} and $C$ is a positive constant depending on the a-priori parameters.
\end{prop}
\begin{proof}
According to \eqref{eq:cgo} and \eqref{eq:cgo equality}, the integral identity \eqref{ineq:core identity} follows directly from \eqref{eq:intergal identity} by utilizing \eqref{eq:rho expan} and \eqref{eq:u expan}.

To estimate each term in the right-hand side of \eqref{ineq:core identity}, let us recall the incomplete Gamma functions $\gamma(\cdot,\cdot)$ and $\Gamma(\cdot,\cdot)$ (cf. \cite{Blasten2020}) which are defined by
\begin{align*}
\gamma:\mathbb{R}_+\times\mathbb{R}_+\rightarrow\mathbb{R}:\,\, (s,x)\mapsto \gamma(s,x),\quad \Gamma:\mathbb{R}_+\times\mathbb{R}_+\rightarrow\mathbb{R}:\, (s,x)\mapsto \Gamma(s,x),
\end{align*}
where
\begin{align}\label{defi:gamma}
\gamma(s,x)=\int_0^x e^{-t} t^{s-1}\mathrm{dt},\quad \Gamma(s,x)=\int_{x}^\infty t^{s-1} e^{-t}\mathrm{dt}.
\end{align}
From the fact $e^{-t}\leq e^{-t/2}e^{-x/2}$ and a variable substitution $t'=t/2$, it is easy to get the estimation:
\begin{align}\label{ineq:gamma}
\gamma(s,x)\leq \Gamma(s)\leq \lceil s-1\rceil,\quad \Gamma(s,x)\leq 2^s\Gamma(s)e^{-x/2},
\end{align}
where $\lceil s \rceil $ the largest integer satisfying $\lceil s \rceil\leq s$, and   $\Gamma(s)$ is the complete Gamma function of $s$ bounded by $\lceil s \rceil!$.

To estimate the first four terms in the right-hand side of \eqref{ineq:core identity}, in the following we derive some important asymptotic  inequalities with respect to the parameter  $|\Re \xi|$ in the CGO solution $\bu_0$. By using the polar coordinate transformation, as well as \eqref{ineq:delta_0}, \eqref{defi:gamma} and \eqref{ineq:gamma}, we can obtain
 \begin{align}\label{ineq:1-1'}
\Big| \int_{\Acal\backslash S_h} e^{\xi\cdot (\bx-\bx_0)}\mathrm{dx}\Big|&\leq \Big| \int_{\Acal\backslash S_h} e^{-\delta_0|\Re\xi||\bx-\bx_0|}\mathrm{dx}\Big|\leq\sigma(\mathbb{S}^{n-1})\int^\infty_{h}e^{-\delta_0|\Re\xi|r}r^{n-1}\mathrm{dr}\nonumber\\
&\leq \sigma(\mathbb{S}^{n-1})\int_{h\delta_0|\Re\xi|r}^\infty e^{-t}t^{n-1}(\delta_0|\Re\xi|)^{-n}\mathrm{dt}\nonumber\\
 &= \sigma(\mathbb{S}^{n-1})(\delta_0|\Re\xi|)^{-n}\Gamma(n,\delta_0|\Re\xi|h)\nonumber\\
 &\leq C_{\delta_0,n}|\Re\xi|^{-n}e^{-\delta_0|\Re\xi|\frac{h}{2}}
    \end{align}
as $|\Re \xi| \rightarrow +\infty $.  Similarly, we have
 \begin{align}\label{ineq:1-2'}
\Big| \int_{ S_h} e^{\xi\cdot (\bx-\bx_0)}\mathrm{dx}\Big|&\leq \Big| \int_{ S_h} e^{-\delta_0|\Re\xi||\bx-\bx_0|}\mathrm{dx} \leq\sigma(\mathbb{S}^{n-1})\int^h_{0}e^{-\delta_0|\Re\xi|r}r^{n-1}\mathrm{dr}\nonumber\\
&\leq \sigma(\mathbb{S}^{n-1})\int^{h\delta_0|\Re\xi|r}_0 e^{-t}t^{n-1}(\delta_0|\Re\xi|)^{-n}\mathrm{dt}\nonumber\\
&\leq \sigma(\mathbb{S}^{n-1})(\delta_0|\Re\xi|)^{-n}\gamma(n,\delta_0|\Re\xi|h)\nonumber\\
&\leq C'_{\delta_0,n}|\Re\xi|^{-n}
    \end{align}
as $|\Re \xi| \rightarrow +\infty $, where $\sigma(\mathbb{S}^{n-1})$ is the measure of $\mathbb{S}^{n-1}$.  For scalar functions $f$ and $g$, which admit $|f|\leq A|\bx-\bx_0|^B$ with  $A$ and $B$  being positive constants,  and $g\in L^q(S_h)$ with $1/p+1/q=1$. Combining the H\"{o}lder inequality, \eqref{ineq:delta_0}, \eqref{defi:gamma} with \eqref{ineq:gamma}, we obtain
\begin{align}\label{ineq:1-2}
\Big|  \int_{S_h} e^{\xi\cdot (\bx-\bx_0)}f\,g\mathrm{dx}\Big|&\leq A\Big(\int_{S_h} e^{\Re\xi\cdot(\bx-\bx_0)p}|\bx-\bx_0|^{Bp}\mathrm{dx} \Big)^{\frac{1}{p}}\,\|g\|_{L^{q}(S_h)}\nonumber\\
&\leq A\,\sigma(\mathbb{S}^{n-1})^{\frac{1}{p}}(\delta_0p|\Re\xi|)^{-B-n/p}\big(\gamma(Bp+n,\,\delta_0 p|\Re\xi|h)\big)^{\frac{1}{p}}\,\|g\|_{L^{q}(S_h)}\nonumber\\
&\leq C_{A,B,\delta_0,p}|\Re\xi|^{-B-n/p}\,\|g\|_{L^{q}(S_h)}
\end{align}
as $|\Re \xi| \rightarrow +\infty $. By similar arguments, we can derive the following estimations  for vector-type functions $\bff$  and $\bg$ satisfying $|\bff|\leq A|\bx-\bx_0|^B$  with  $A$ and $B$  being positive constants and $\bg\in L^q(S_h)^n$ with $1/p+1/q=1$,
\begin{align}\label{ineq:1-3}
\Big|  \int_{S_h} e^{\xi\cdot (\bx-\bx_0)}\,\bff\cdot \bg\,\mathrm{dx}\Big|&\leq A\Big(\int_{S_h} e^{\Re\xi\cdot(\bx-\bx_0)p}|\bx-\bx_0|^{Bp}\mathrm{dx} \Big)^{\frac{1}{p}}\,\,\|\bg\|_{L^q(S_h)^n}\nonumber\\
&\leq A\,\sigma(\mathbb{S}^{n-1})^{\frac{1}{p}}(\delta_0p|\Re\xi|)^{-B-n/p}\gamma(Bp+n,\,\delta_0 p|\Re\xi|h)^{\frac{1}{p}}\,\,\|\bg\|_{L^q(S_h)^n}\nonumber\\
&\leq C'_{A,B,\delta_0,p}|\Re\xi|^{-B-n/p}\,\|\bg\|_{L^q(S_h)^n}
\end{align}
as $|\Re \xi| \rightarrow +\infty $.

By virtue of  \eqref{eq:rho expan} and \eqref{eq:u expan}, from \eqref{ineq:1-1'}-\eqref{ineq:1-3}, we can deduce that
\begin{subequations}
\begin{align}
&R_1:=\Big|\int_{\Acal\backslash S_h} e^{\xi\cdot (\bx-\bx_0)}\eta\cdot \bu(\bx_0) \mathrm{dx}\Big |\leq C_{\delta_0,n}|\bu(\bx_0)||\Re\xi|^{-n}e^{-\delta_0|\Re\xi|\frac{h}{2}},\label{estima:1-1}\\
&R_2:=\Big|\int_{S_h} e^{\xi\cdot (\bx-\bx_0)}\eta\cdot \bu(\bx_0) \mathrm{dx}\Big |\leq C'_{\delta_0,n}|\bu(\bx_0)||\Re\xi|^{-n},\label{estima:1-1'}\\
&R_3:=\Big|\int_{S_h}e^{\xi\cdot(\bx-\bx_0)} (\rho_1+1)\,\eta\cdot \bu(\bx_0) \mathrm{dx} \Big|\leq C_{\Rcal,n,\delta_0}|\bu(\bx_0)|\big|\Re\xi\big|^{-n-\theta_1},\label{estima:1-2}\\
&R_4:=\Big| \int_{S_h}e^{\xi\cdot(\bx-\bx_0)}(\rho-1)\eta\cdot\bu_1\mathrm{dx}\Big|\leq C_{\Mcal,\Rcal,n,\delta_0}|\Re\xi|^{-n-\theta_2}\label{estima:1-3}.
\end{align}
\end{subequations}
%

For the boundary integral term in the right-hand side of \eqref{ineq:core identity}, we divide the integration area $\partial D_h$ into two parts $\partial \Bcal\cap B_h(\bx_0)$ and ${\mathcal S}:= \partial S_h(\bx_0) \backslash  ( \partial \Bcal\cap B_h(\bx_0))  $. For the boundary integral over $\partial \Bcal\cap B_h(\bx_0)$, the Cauchy-Schwarz inequality implies that
{
\begin{align}\label{ineq:first part 1_1}
I_1:&=\Big|\int_{\partial\Bcal\cap B_h(\Bx_0)} \bu_0\cdot T_{\hat{\nu}}(\bu'-\bu)-(\bu'-\bu)T_{\hat{\nu}}(\bu_0)  \mathrm{d\sigma}\Big|\nonumber\\
&\leq \sqrt{\sigma(\partial\Bcal\cap B_h(\bx_0))}\,\,\Big(\int_{\partial\Bcal\cap B_h(\bx_0)}\big|\bu_0\cdot T_{\hat{\nu}}(\bu'-\bu)-(\bu'-\bu)T_{\hat{\nu}}(\bu_0)\big|^2\mathrm{d\sigma}\Big)^{1/2}\nonumber\\
&\leq\sqrt{\sigma(\partial\Bcal\cap B_h(\bx_0))}\,\Big[\int_{\partial\Bcal\cap B_h(\bx_0)}\big(| \bu_0|+|T_{\hat{\nu}}(\bu_0)|\big)^2\mathrm{d\sigma}\,\Big]^{1/2}\,T_1,
\end{align}}where $T_1=\sup\limits_{\partial \Bcal\cap B_h(\bx_0)}\big\{ |T_{\hat{\nu}}(\bu'-\bu)|,|\bu-\bu'|\big\}.$
By virtue of \eqref{eq:surface traction} and \eqref{eq:cgo}, one can directly derive that
\begin{align*}
T_{\hat{\nu}}(\bu_0)=\mu e^{\xi\cdot (\bx-\bx_0)}\big[(\xi\cdot \hat{\nu})\, \eta+(\eta\cdot \hat{\nu})\,\xi\big]\quad\mbox{and}\quad \big|  T_{\hat{\nu}}(\bu-\bu')\big|\leq C_{\lambda,\mu}|\nabla (\bu-\bu')|,
\end{align*}
where $\nu$ is the outward unit normal vector to $\partial \Bcal\cap B_h(\bx_0)$ and $C_{\lambda,\mu}$ is a positive constant depending on $\lambda$ and $\mu$. From \eqref{ineq:first part 1_1}, we obtain that
{
\begin{align}\label{ineq:first part 1_2}
\quad\quad \quad\quad I_1&\leq\sqrt{\sigma(\partial\Bcal\cap B_h(\bx_0))}\,\Big[\int _{\partial\Bcal\cap B_h(\bx_0)}\Big( |e^{\xi\cdot (\bx-\bx_0)}| +\big|\mu e^{\xi\cdot (\bx-\bx_0)}(\xi\cdot \hat{\nu})\eta \notag \\
&\hspace{5cm} +(\eta\cdot \hat{\nu})\xi\big| \big)^2\mathrm{d\sigma}\Big]^{1/2}\,T_1\nonumber\\
&\leq \sqrt{\sigma(\partial\Bcal\cap B_h(\bx_0))}\, \Big[\int _{\partial\Bcal\cap B_h(\bx_0)}\big( |\eta|+\mu|\eta(\xi\cdot \hat{\nu})+\xi(\eta\cdot \hat{\nu})|    \big)^2\mathrm{d\sigma}\Big]^{1/2}\,T_1\nonumber\\[1mm]
&\leq C'_{\lambda,\mu}\,\,\,\sigma(\partial\Bcal\cap B_h(\bx_0))\,\sup\limits_{\partial\Bcal\cap B_h(\bx_0)}\big\{ |\nabla(\bu-\bu')|,|\bu-\bu'|\big\},
\end{align}}where $C'_{\lambda,\mu}$ is a positive constant depending on $\lambda$ and $\mu$.
There exist finite half-spaces $\big\{H_j\big\}^N_{j=1}$ through $\bx_0$ with a codimension of $1$ such that $$\partial \Bcal\cap B_h(\bx_0)\subset \bigcap\limits^{N}_{j=1}\big( H_j\cap B_h(\bx_0)\big).$$
Hence, we have
\begin{equation}\label{eq:sigma line}
	\sigma (\partial\Bcal\cap B_h(\bx_0))\leq C_{N,n}\sigma (\partial H_j\cap B_h(\bx_0)) \leq C_{N,n}\,\,h^{n-1},
\end{equation}
where $C_{N,n}$ is a positive constant depending only $n$ and $N$. Substituting \eqref{eq:sigma line} into \eqref{ineq:first part 1_2}, one has
\begin{equation}\label{eq:423}
	I_1 \leq C h^{n-1} \sup\limits_{\partial\Bcal\cap B_h(\bx_0)}\big\{ |\nabla(\bu-\bu')|,|\bu-\bu'|\big\},
\end{equation}
where $C$ is a a-priori positive constant depending on $\lambda$, $\mu$, $N$ and $n$.

For the boundary integral over ${\mathcal S}=\partial S_h(\bx_0) \backslash  ( \partial \Bcal\cap B_h(\bx_0)) $, by the Cauchy-Schwarz inequality and the fact that $\sigma(\Bcal\cap S_h(\bx_0))\leq C h^{n-1}$, where $C$ is a positive constant depending on $n$ and $\pi$, one has
\begin{align}\label{eq:424}
\quad\quad\quad\quad I_2:&=\Big| \int_{{\mathcal S}} \bu_0\,T_{\hat{\nu}}(\bu-\bu') -(\bu-\bu')\,T_{\hat{\nu}}(\bu_0) \mathrm {d\sigma}\Big|\nonumber\\
&\leq \sqrt{\sigma({\mathcal S})}\,\,\big(\int_{{\mathcal S}}\big|\bu_0\cdot T_{\hat{\nu}}(\bu'-\bu)-(\bu'-\bu)T_{\hat{\nu}}(\bu_0)\big|^2\mathrm{d\sigma}\,\big)^{1/2}\,\,T_2\nonumber\\
&\leq \sqrt{\sigma({\mathcal S})}\,\,\big(\int_{{\mathcal S}}(\big|\bu_0|+ |T_{\hat{\nu}}(\bu_0)|\big)^2\mathrm{d\sigma}\,\big)^{1/2}\,\,T_3\nonumber\\
&\leq C'_{\lambda,\mu,n}\,\,\sigma({\mathcal S})\,\,e^{-\delta_0|\Re\xi|h}\,\,T_3\nonumber\\[2pt]
&\leq C'_{\lambda,\mu,n} \,\,h^{n-1}\,e^{-\delta_0|\Re\xi|h}\,\,T_4,
\end{align}
where
\begin{align*}
T_2&=\sup\limits_{{\mathcal S}}|\nabla(\bu-\bu')|+\sup\limits_{{\mathcal S}}|\bu-\bu'|,\\[2pt]
T_3&=\|\bu_\mathrm s-\bu'_\mathrm s\|_{C^1(\overline{B_h(\bx_0)})^n} +\big\|\bu_\mathrm p-\bu'_\mathrm p\|_{C^1(\overline{B_h(\bx_0)})^n},\\[5pt]
T_4&=\|\bu_\mathrm s\|_{H^2(B_{2R})^n}+\|\bu_\mathrm p\|_{H^2(B_{2R})^n}+\|\bu'_\mathrm s\|_{H^2(B_{2R})^n}+\|\bu'_\mathrm p\|_{H^2(B_{2R})^n}.
\end{align*}
Finally, it is easy to see that
\begin{equation}\label{eq:425}
	C\big|(\rho(\bx_0)-1)\int_{\Acal} e^{\xi\cdot (\bx-\bx_0)}\eta\cdot \bu(\bx_0)\mathrm{dx}\big|\leq R_1+R_2+R_3+I_1+I_2,
\end{equation}
where $C=C\big(\kappa_\mathrm p,\kappa_\mathrm s, \Mcal, \Rcal, \theta_1,\theta_2, n,\delta_0\big)>0$. Substituting \eqref{estima:1-1} to \eqref{estima:1-3}, \eqref{eq:423} and \eqref{eq:424} into \eqref{eq:425}, we prove \eqref{ineq:core estimate}.
\end{proof}

\section{Proof of Theorem \ref{th:main1} }\label{sec:proTh1}
In this section, we shall give the proof of Theorem \ref{th:main1}.

\begin{proof}[Proof of Theorem \ref{th:main1}]
 Let $\bx_0$ be a vertex of $\Omega$ such that $\mathrm{dist}(\bx_0,\Omega')=\mathrm{dist}_{\Hcal}(\Omega,\Omega')$ and $Q$ be the convex hull of $\Omega$ and $\Omega'$. From Lemma \ref{lem:convex hull_2D} and Lemma \ref{lem:convex hull 3D}, we know $\bx_0$ must be a vertex of $Q$. Let $\Bcal$ be the open polyhedral cone generated by $Q$ at $\bx_0$.\,Then there must be an open spherical cone $\Dcal$ whose opening angle at $\bx_0$ is in $(2\alpha_M,\,\pi)\subset(2\alpha_m,\, \pi)$. Let $\Acal$ denote the open polyhedral cone generated by $\Omega$ at $\bx_0$ and $\Ccal$ be an open spherical cone such that $(\Acal,\Ccal)\in \Fcal(\alpha_m,\alpha_M,n)$ (cf. Definition \ref{def:cone_pair}). Denote $\hbar=\mathrm{dist}(\bx_0,\Omega')<1$ and choose $h\leq\min\{l_0,\hbar\}$ such that
 $$
  \begin{cases}
  \Delta^*\bu'+\omega^2\bu'={\bf 0} \quad \mbox{in}\quad S_h\subset D_h,\\
  S_h\cap \Omega'=D_h\cap \Omega'=\emptyset,
  \end{cases}
 $$
where $S_h=\Omega\cap B_h(\bx_0)=\Acal\cap B_h(\bx_0)$ and $D_h=Q\cap B_h(\bx_0)=\Bcal\cap B_h(\bx_0)$.
From the proof of Proposition \ref{pro:estimate on boundary}, we can derive that
 $$
 \|\bu_\mathrm \alpha\|_{H^2(B_{2R})^n},\,\|\bu'_\mathrm \alpha\|_{H^2(B_{2R})^n}\leq C_{\kappa_\mathrm p,\kappa_\mathrm s, R,\Scal}, \quad \mathrm\alpha=\mathrm p,\mathrm s.
 $$
 Thus
 $$
 \bu\in H^2(B_{2R})^n\,\, \mbox{and }\,\, \bu'\in H^2(B_{2R})^n.
 $$
 The a-priori admissibility assumptions of $\rho$ and the local $H^2$-regularity of $\bu$ near $S_h$ imply the following splittings
 \begin{align*}
&\rho(\bx)=1+\rho(\bx_0)+\rho_1(\bx),\quad|\rho_1(\bx)|\leq \|\rho_1\|_{C^{\theta_1}( K\cap B_h(\bx_0) )}|\bx-\bx_0|^{\theta_1},\,\,\theta_1\in(0,1),\\
&\bu(\bx)=\bu(\bx_0)+\bu_1(\bx),\,\,\quad\quad |\bu_1(\bx)|\leq \|\bu_1\|_{C^{\theta_2}( K\cap B_h(\bx_0) )^n} |\bx-\bx_0|^{\theta_2},\,\,\theta_2\in(0,1),
\end{align*}
where we use the Sobolev embedding theorem from $H^2$ to $C^{\theta_2}$.
 So the conditions of Proposition \ref{prop:important estimate} are all satisfied. We absorb all the constants into the left-hand side, which depend only on the {\it a-priori} parameters, then \eqref{ineq:core estimate} reduces to
\begin{align}\label{ineq:core estimate2}
C\big|(\rho(\bx_0)-1)\int_{\Acal} e^{\xi\cdot (\bx-\bx_0)}\eta\cdot \bu(\bx_0)\mathrm{dx}\big|&\leq |\Re\xi|^{-n}e^{-\delta_0|\Re\xi|h/2}+|\Re\xi|^{-n-\min\{\theta_1,\theta_2\}}\nonumber\\
&+h^{n-1}\sup_{\partial \Bcal\cap B_h(\bx_0)}\Big( |\bu-\bu'|+|\nabla \bu-\nabla \bu'|  \Big)^{-1/2}\nonumber\\
&+h^{n-1}e^{-\delta_0|\Re\xi|h}.
\end{align}
From Proposition \ref{pro:estimate on boundary} and the fact that $\partial Q\supset\partial\Bcal\cap B_h(\bx_0)$, we have
\begin{align*}
\sup_{\partial\Bcal\cap B_h(\bx_0)}\Big( |\bu-\bu'|+|\nabla \bu-\nabla \bu'|  \Big)\leq C'\left(\ln\ln (\Ncal/\varepsilon)\right)^{-1/2}.
\end{align*}
Let $\small \delta(\varepsilon)= \left(\ln\ln (\Ncal/\varepsilon)\right)^{-1/2}$, where $C'$ depends only on those {\it a-priori} parameters.\\
On the other hand, note that $(\Acal,\Ccal)\in\Fcal(\alpha_m,\alpha_M,n)$. We can apply Proposition \ref{prop:cgo} to obtain an estimate of the term on the left-hand side of \eqref{ineq:core identity} as follows
\begin{equation}\label{eq:lowerbound+CGO}
\Big| \int_{\Acal} e^{\xi(\tau)\cdot(\bx-\bx_0)}\,\eta\cdot \bu(\bx_0)\mathrm dx        \Big|\geq C'_{\Acal}\tau^{-n},
\end{equation}
where $C'_{\Acal}=C'_{\Acal}(\Acal,\bp,\bu(\bx_0))>0$.

From Definition \ref{defin:admissible class}, we have $|\rho(\bx_0)-1|\neq 0$. Combining \eqref{ineq:core estimate2} with \eqref{eq:lowerbound+CGO}, we have
\begin{align*}
C_1\big|(\rho(\bx_0)-1)\big|&\leq e^{-\delta_0|\Re\xi|h/2}+\tau^{-\min\{\theta_1,\theta_2\}}+h^{n-1}\tau^n\delta(\varepsilon)+h^{n-1}\tau^n e^{-\delta_0|\Re\xi|h},
\end{align*}
where $C_1$ is a positive constant depending only on those {\it a- priori} parameters and $\bu(\bx_0)$.
Using the fact that $e^{-\bx}\leq \frac{1}{\bx}$ and $e^{-\bx}\leq\frac{(n+4)!}{\bx^{n+4}}$ for all $\bx>\bf0$, we get
\begin{align*}
C_2\big|(\rho(\bx_0)-1)\big|&\leq \tau^{-1}h^{-1}+\tau^{-\min\{\theta_1,\theta_2\}}+h^{n-1}\tau^{n+1}\delta(\varepsilon)+h^{-5}\tau^{-3} \\
&\leq C_0(\tau^{-m}h^{-5}+\tau^{n+1}\delta(\varepsilon)h^{n-1})\\
&\leq C_0 h^{n-1}[\tau^{-m}h^{-n-4}+\tau^{n+1}\delta(\varepsilon)].
\end{align*}
Here, $m=\min\{1, \theta_1,\theta_2\}$. Note that $\tau^4>1$, $h^{-n}>1$ and $0<h^{n-1}<1$. Dividing the both sides of the above inequality by $h^{n-1}$, we obtain the following inequality:
\begin{equation}\label{eq:importan ineq8-1}
C_3\big|(\rho(\bx_0)-1)\big|\leq \tau^{-m}h^{-n-5}+\tau^{n+5}\delta(\varepsilon).
\end{equation}
To determinate a minimum modulo constants of the right-hand side of \eqref{eq:importan ineq8-1}. Set
\begin{equation}\label{eq:tau_0}
\tau=\tau_e=\Big[\frac{1}{h^{n+5}\delta(\varepsilon)}\Big]^{\frac{1}{m+n+1}},
\end{equation}
we have
\begin{equation}\label{eq:importan ineq8-5}
C_4\big|(\rho(\bx_0)-1)\big|\leq 2\,\delta(\varepsilon)^{\frac{m}{m+n+5}}\,h^{-\frac{(n+5)^2}{m+n+5}}.
\end{equation}
Note that $\tau_e\geq\tau_e \,h^{\frac{n+5}{m+n+5}}=\delta(\varepsilon)^{-\frac{1}{m+n+5}}\geq\tau_0$ and $h\leq\hbar\leq l_0$ for sufficiently small $\varepsilon$.
From \eqref{eq:importan ineq8-5}, we have
\begin{align*}
\hbar&\leq C_5\,\delta(\varepsilon)^{\frac{m}{m+n+5}}\big|\rho(\bx_0)\big|^{-\frac{m}{(n+5)^2}-\frac{1}{n+5}}
\leq C_5 \, \delta(\varepsilon)^{\frac{m}{m+n+5}}\big|\epsilon_0+1\big|^{-\frac{m}{(n+5)^2}-\frac{1}{n+5}}\\
&\leq C_6\,\big(\ln\ln(\Ncal/\varepsilon)\big)^{-\frac{m}{2(n+5)^2}},
\end{align*}
 and the claim readily follows.

 The proof is complete.
\end{proof}

\section{Related results about invisibility}\label{sec4}

%
%

The present section is devoted to the technical details of Theorem~\ref{thm:main2summary}.

\begin{thm}\label{th:main2}
Let $(\Omega; \rho)$ be an {admissible polyhedral scatterer}. The plane incident wave $\bu^i$ of the form \eqref{eq:incident wave} is a nontrivial entire solution to \eqref{eq:navier equation} with $\rho=1$. Let $\bU$ be the far-field pattern of the scattered wave $\bu^s$ by medium scatterer $(\Omega;\rho)$. Then there exists a positive constant $C$ only depending on the {\it a-priori} parameters such that the following estimate holds:
\begin{align}\label{eq:lower bound of Main result}
\|\bU\|_{L^2(\mathbb{S}^{n-1},\mathbb{C}^n\times \mathbb{C}^n)}\geq C>0.
\end{align}
\end{thm}
In fact, $\Omega\subset \mathbb{R}^{n}$, $n=2,3$, can be relaxed to a general bounded Lipschitz domain, which admits a convex polygonal/polyhedral point on its boundary, namely, there exits $\bx_c\in \partial \Omega$ such that $\Omega\cap B_l(\bx_c)\subset \mathbb{R}^2$ is a plane sector for $l_0>l>0$ and $\Omega\cap B_l(\bx_c)\subset \mathbb{R}^3$ is a polyhedral cone with a spherical bottom surface for $l_0>l>0$. The general result is stated as follows.
\begin{thm}\label{th:main2'}
Let $\Omega\subset \mathbb{R}^{n}$, $n=2,3$, be a bounded Lipschitz domain, with a convex polygonal/polyhedral point $\bx_c\in \partial \Omega$. Let $\rho(\bx)$ be a uniformly $\theta$-H\"{o}lder continuous function in $\Omega$, $0<\theta\leq1$. The plane incident wave $\bu^i$ of the form \eqref{eq:incident wave} is a nontrivial entire solution to \eqref{eq:navier equation} with $\rho=1$. Let $\bU$ be the far-field pattern of the scattered wave $\bu^s$ by medium scatterer $(\Omega;\rho)$. Suppose that the following conditions are fulfilled:
\begin{itemize}
\item[{\rm (1)}] $\Omega=\mathrm{supp}(1-\rho)\subset \overline{B_R}$ for some $R>l>0$;
\item[{\rm (2)}]  $\|\bu^s\|_{H^2(B_R)^n}\leq \Ncal$ for $\Ncal>0$;
\item[{\rm (3)}] $|\rho(\bx_c)-1| \geq\epsilon_0>0$ and $\|\rho\|_{C^{\theta}(\Omega)}\leq \Mcal$ for some $\Mcal>0$;
\item[{\rm (4)}] For $n=2$, $0<2\alpha_m\leq \alpha\leq 2\alpha_M \leq \pi$, where $\alpha$ signifies the opening angle at $\bx_c$.
\end{itemize}
Then the following estimate holds:
\begin{align}\label{eq:lower bound of Main result'}
\|\bU\|_{L^2(\mathbb{S}^{n-1},\mathbb{C}^n\times \mathbb{C}^n)}\geq C>0,
\end{align}
where $C$ is a positive constant, which depends on the {\it a-priori} parameters $\theta, l, \Ncal, \Mcal, \epsilon_0$, and $R$.
\end{thm}

\begin{proof}
We follow similar arguments in the proof of Theorem \ref{th:main1} with some necessary modifications.
We take $\Omega'=\emptyset$ and $\rho'\equiv1$ in $\mathbb{R}^n$. Clearly, the incident wave $\bu^i$ of the form \eqref{eq:incident wave} is an entire solution to $\Delta^*\bu'+\omega^2 \bu'=\bf 0$, which yields $\bu'^s\equiv\bf0$ and $\bU'\equiv\bf0$. Denote $\varepsilon=\|\bU\|_{L^2(\mathbb{S}^{n-1},\mathbb{C}^n\times \mathbb{C}^n)}$. In fact, we only consider the effect caused by $\bx_c$ in its micro-local area $\Omega\cap B_h(\bx_c)$, and therefore we can reduce the corresponding reasoning in the proof Theorem \ref{th:main1}.
Here we take $h=l$. After a series of derivations similar to Theorem \ref{th:main1}, we can obtain the following inequality
\begin{align}\label{ineq:8-5'}
C_1\,\big|(\rho(\bx_c)-1)\big|\leq \tau^{-m}\,l^{-5}+\tau^{n+1}\delta(\varepsilon)\,l^{n-1},
 \end{align}
 where $\delta(\varepsilon)=\left(\ln\ln (\Ncal/\varepsilon )\right)^{-1/2}$ and $C_1>0$ does not depend on $l$.

Set
$$\tau=\tau_e=\Big[\frac{1}{l^{n+5}\delta(\varepsilon)}\Big]^{\frac{1}{m+n+1}},$$
we have
\begin{align}\label{ineq:8-5555'}
C_2\big|(\rho(\bx_c)-1)\big|\leq 2\,\delta(\varepsilon)^{\frac{m}{m+n+5}}\,l^{-\frac{(n+5)^2}{m+n+5}}.
 \end{align}
It is straightforward to verify that for sufficient small $\varepsilon$, one has
$$\tau_e\geq \tau_e\, l^{\frac{n+5}{m+n+5}}=\frac{C_3}{\delta(\varepsilon)^{m+n+5}}=C_3(\ln\ln(\Ncal/\varepsilon))^{\frac{(n+5)^2}{2(m+n+5)}}\geq \tau_0,$$
Using $\tau=\tau_e$ in \eqref{ineq:8-5555'}, $ |\rho(\bx_c)-1|\geq\epsilon_0>0$ and $\delta(\varepsilon)=\left(\ln\ln (\Ncal/\varepsilon )\right)^{-1/2}$,
\begin{align}\label{ineq:solve h}
\varepsilon&\geq \Ncal\Big[\exp\exp \big\{  C_3 \,l^{-\frac{(n+5)^2}{m}}|(\rho(\bx_c)-1)|^{-\frac{m+n+5}{m}} \big\} \Big]^{-1}\nonumber\\
&\geq\Ncal\Big[\exp\exp \big\{  C_3 \,l^{-\frac{(n+5)^2}{m}}\epsilon_0^{-\frac{m+n+5}{m}} \big\} \Big]^{-1}.
\end{align}
That is,
\begin{align*}
\|\bU\|_{L^2(\mathbb{S}^{n-1},\mathbb{C}^n\times \mathbb{C}^n)}\geq C>0,
\end{align*}
where $C=\Ncal\Big[\exp\exp \big\{  C_3 \,l^{-\frac{(n+5)^2}{m}}\epsilon_0^{-\frac{m+n+5}{m}} \big\} \Big]^{-1}$.

The proof is complete.
\end{proof}

\section*{Acknowledgement}

 The research of Z Bai was partially supported by the National Natural Science Foundation of China under grant 11671337 and the Natural Science Foundation of Fujian Province of China under grant 2021J01033.  The work of H Diao was supported in part by the startup fund from Jilin University.  The work of H Liu was supported by the startup fund from City University of Hong Kong and the Hong Kong RGC General Research Fund (projects 12301420, 11300821 and 12301218), and the NSFC/RGC Joint Research Fund (project N\_CityU101/21).

\end{document}